\documentclass[11pt]{article}

\usepackage[letterpaper,margin=1in]{geometry}
\usepackage{amsmath,amssymb,amsthm,mathtools}
\usepackage{enumitem}
\usepackage{microtype}
\usepackage{cite}
\usepackage{xcolor}
\usepackage[hidelinks]{hyperref}

\allowdisplaybreaks
\setlist[enumerate]{leftmargin=2.3em,itemsep=2pt,topsep=4pt}

\newtheorem{theorem}{Theorem}[section]
\newtheorem{proposition}[theorem]{Proposition}
\newtheorem{lemma}[theorem]{Lemma}
\newtheorem{corollary}[theorem]{Corollary}
\theoremstyle{definition}
\newtheorem{definition}[theorem]{Definition}
\theoremstyle{remark}

\newcommand{\R}{\mathbb{R}}
\newcommand{\cl}{\operatorname{cl}}
\newcommand{\dom}{\operatorname{dom}}
\newcommand{\dist}{\operatorname{dist}}
\newcommand{\range}{\operatorname{range}}

\newcommand{\supp}{\operatorname{supp}}
\newcommand{\Diag}{\operatorname{Diag}}

\newcommand{\ri}{\operatorname{ri}}
\newcommand{\cir}{\operatorname{cir}}
\newcommand{\1}{\mathbf{1}}
\newcommand{\dd}{\,\mathrm{d}}
\newcommand{\ip}[2]{\left\langle #1,#2\right\rangle}
\newcommand{\norm}[1]{\left\lVert #1\right\rVert}
\newcommand{\abs}[1]{\left\lvert #1\right\rvert}
\newcommand{\cC}{\mathcal{C}}
\newcommand{\cN}{\mathcal{N}}
\newcommand{\transpose}{\mathsf{T}}

\title{Establishing Boundary KKT Convergence of Mirror Descent
through Reparameterization}
\author{
Kuangyu Ding\thanks{School of Industrial Engineering, Purdue University,
West Lafayette, IN 47907, USA. Email: \nolinkurl{kuangyud@u.nus.edu}.
The work of Ding was supported by the Office of Naval Research (ONR) under Grant N00014-24-1-2751.}
\and
Kim-Chuan Toh\thanks{Department of Mathematics, National University of Singapore, Singapore
119076. Email: \nolinkurl{mattohkc@nus.edu.sg}.
The research of this author is supported by the Ministry of
Education, Singapore, under its Academic Research Fund Tier 2 grant MOE-T2EP20224-0017.
}
}
\date{\today}

\begin{document}
\maketitle

\begin{abstract}
Sequence convergence to a boundary Karush--Kuhn--Tucker (KKT) point has long
remained unclear for nonconvex mirror descent with Legendre kernels.  The difficulty arises from the blow-up of the gradient of the Legendre kernel at the boundary.  Recent work~\cite{dingtoh2026nonkkt} shows that mirror descent can accumulate at
non-KKT boundary points despite decreasing objective values, precluding a
convergence guarantee to KKT points in general. Despite this negative result, mirror descent remains effective in many
real applications.  Motivated by this contrast, we address the boundary
difficulty directly and establish KKT convergence of mirror descent for a broad
class of structured nonconvex problems.  We analyze mirror descent in
reparameterized variables, where the Hessian metric is flattened and remains
nondegenerate as the boundary is approached.  Under extension and definability conditions jointly
coupling the objective, the Legendre kernel, and the feasible region, the
reparameterized sequence has finite length and converges, thereby recovering
convergence to a KKT point of the original sequence.
Our general framework applies to some concrete instances: Shannon entropy, Fermi--Dirac entropy, and power kernels on polyhedron.
\end{abstract}


\section{Introduction}

We consider the following smooth nonconvex problem with linear
constraints:
\begin{equation}\label{eq:problem}
 \min_{x\in X} f(x),
 \qquad
 X:=\overline C\cap L,
 \qquad
 L:=\{x\in\R^n:Ax=b\},
\end{equation}
where \(C:=\prod_{i=1}^n I_i\) with
each \(I_i=(\ell_i,u_i)\) being a possibly unbounded open
interval, and {$\overline{C}$ denotes the closure of $C$.}  Here \(A\in\R^{m\times n}\) has full row rank
and \(b\in\R^m\).  We assume that
\(X^\circ:=C\cap L\neq\varnothing\), so \(X^\circ=\ri X\).  The objective \(f\) is \(C^1\) on a neighborhood
of \(X\).  By the sum rule of normal cones,
\(N_X(x)=N_{\overline C}(x)+N_L(x)\) for every \(x\in X\).
Accordingly, \(x_\star\in X\) is a KKT point of \eqref{eq:problem} if and only if
there exists \(\lambda_\star\in\R^m\) such that
\begin{equation}\label{eq:KKT}
 0\in\nabla f(x_\star)+A^\transpose\lambda_\star+N_{\overline C}(x_\star).
\end{equation}
The vector \(\nabla f(x)+A^\transpose\lambda\) is the
\emph{Lagrangian gradient} associated with \(\lambda\). Consider mirror descent generated by a kernel from the following standard
class.

\begin{definition}[Legendre kernel {\cite[Section~26]{rockafellar1970convex}}]
\label{def:legendre-kernel}
A proper closed convex function \(\phi:\R^n\to(-\infty,+\infty]\) is essentially smooth if \(\operatorname{int}(\operatorname{dom}\phi)\neq\varnothing\), it is differentiable throughout \(\operatorname{int}(\operatorname{dom}\phi)\), and \(\norm{\nabla\phi(x^\nu)}\to\infty\) for every sequence \(x^\nu\in\operatorname{int}(\operatorname{dom}\phi)\) converging to a point on the boundary of $\operatorname{dom}\phi$.
It is essentially strictly convex if it is strictly convex on every convex subset of \(\operatorname{dom}\partial\phi\), where $\partial \phi$ denotes the subdifferential of $\phi$.  We call \(\phi\) a Legendre kernel on \(C\) if it has both properties and \(\operatorname{int}(\operatorname{dom}\phi)=C\) and \(\operatorname{cl}(\operatorname{dom}\phi)=\overline C\).
\end{definition}
In this work, we focus on Legendre kernels on \(C\) of the separable
form \(\phi(x):=\sum_{i=1}^n h_i(x_i)\).
Here each \(h_i\) is a one-dimensional Legendre kernel on \(I_i\).
\begin{definition}[Bregman divergence]
For \(x\in C\) and \(y\in\dom\phi\), the Bregman divergence generated by
\(\phi\) is
\[
 D_\phi(y,x):=\phi(y)-\phi(x)-\ip{\nabla\phi(x)}{y-x}.
\]
\end{definition}

Starting from the interior point
\(x_0\in X^\circ\), mirror descent generates
\begin{equation}
 x_{k+1}=T_{\alpha_k}(x_k):=\arg\min_{y\in L}
 \left\{
  \ip{\nabla f(x_k)}{y-x_k}+\frac1{\alpha_k}D_\phi(y,x_k)
 \right\},
 \qquad \alpha_k>0,
 \label{eq:intro-md}
\end{equation}
where the constraint $\overline{C}$ is enforced by the Legendre kernel $\phi$.
The essential smoothness of \(\phi\) ensures that every well-defined mirror
update lies in \(X^\circ\), while accumulation points may lie in the relative boundary.  The
first-order condition for the mirror subproblem yields, for some
\(\lambda_{k+1}\in\R^m\), such that
\begin{equation}\label{eq:intro-dual-increment}
 \frac{\nabla\phi(x_{k+1})-\nabla\phi(x_k)}{\alpha_k}
 +\nabla f(x_k)+A^\transpose\lambda_{k+1}=0,
 \qquad Ax_{k+1}=b.
\end{equation}
Formally replacing the dual increment in \eqref{eq:intro-dual-increment} by a
time derivative suggests the mirror flow
\begin{equation}
 \frac{\dd}{\dd t}\nabla\phi(x(t))
 +\nabla f(x(t))+A^\transpose\lambda(t)=0,
 \qquad Ax(t)=b.
 \label{eq:intro-flow}
\end{equation}
When \(\phi\) is \(C^2\) on \(C\), \eqref{eq:intro-flow} becomes the
Hessian--Riemannian gradient flow on the affine space \(L\)
\cite{alvarez2004hessian,attouch2004singular,bomze2019hessian,boltepauwels2022curiosities,DingToh2025SBSG,dingtoh2025interior}: \(\nabla^2\phi(x(t))\dot x(t)+\nabla f(x(t))+A^\transpose\lambda(t)=0,\; A\dot x(t)=0\),
which can be interpreted as the gradient flow on \(L\) endowed with the
Riemannian metric induced by \(\nabla^2\phi\) on \(C\).

The Bregman geometry is especially natural for optimization over probability
distributions and couplings.  On the simplex, the Shannon entropy kernel
yields normalized multiplicative updates.  On a transportation polytope, the
entropic mirror step is a Kullback--Leibler projection that can be computed by
Sinkhorn scaling~\cite{cuturi2013sinkhorn,peyre2019computational}.  These
projections underlie algorithms in computational optimal transport and in
problems based on the Gromov--Wasserstein distance, including graph matching
and barycenter computation
\cite{memoli2011gromov,peyre2016gromov,li2023gw,rioux2024entropic}.  The
feasible sets contain lower-dimensional faces representing sparse
distributions and couplings, and solutions of interest need not be strictly
positive.  Boundary limits are therefore natural in these problems, not pathological rare cases.

However, convergence of mirror descent for nonconvex problems to a \emph{boundary}
KKT point has long remained unclear.  As the iterates approach the boundary, the
Legendre gradient blows up, so the mirror-step optimality condition cannot be
passed directly to the boundary limit.  Existing work often circumvents this
boundary difficulty by assuming the iterates or solutions of interest to
remain in the interior, modifying the original problem (e.g., by adding
an entropic regularizer or imposing explicit constraint strictly away from the boundary),
or exploiting special problem structure to establish that the solutions of
interest lie in the interior.  Recent
work~\cite{dingtoh2026nonkkt} constructs a counterexample in which a mirror
descent sequence has a non-KKT boundary accumulation point, as detailed in
Section~\ref{sec:intro-boundary-difficulty}.

\subsection{Challenges in boundary KKT recovery}
\label{sec:intro-boundary-difficulty}

A fundamental convergence guarantee for optimization algorithms is that every
accumulation point is KKT stationary.  For Euclidean methods (e.g.
gradient descent, the proximal point method, and the proximal gradient method),
this requirement is typically certified through a generic update
\begin{equation}\label{eq:intro-euclidean-residual}
 x_{k+1}=x_k-\alpha_kG_{\alpha_k}(x_k),
\end{equation}
where \(G_\alpha\) is a stationarity residual.  Under the standard assumptions
for each method, \(G_{\alpha_k}(x_k)\to0\), and the associated optimality
relation can be passed along any convergent subsequence to obtain KKT
stationarity.  By contrast, the Bregman stationarity residuals used for mirror
descent depend on \(\nabla\phi\) and are defined only on the interior \(X^\circ\).
Even if these residuals tend to zero along interior iterates, a boundary
accumulation point need not be KKT stationary.

\medskip
\noindent\textbf{Stationarity measures in mirror descent.}
Existing nonconvex analysis of Bregman methods typically formulates
stationarity through an interior residual.  A basic choice is the Bregman
displacement \(R_\alpha(x):=D_\phi(T_\alpha(x),x)\) of one mirror step,
used in composite optimization and variational-inequality analysis
\cite{DingLiToh2025SBPG,dangLan2015}. Another choice uses a Bregman proximal point of the full
objective~\cite{zhangHe2018}.  For \(\alpha>0\), let
\(p_\alpha(x)\in\arg\min_{y\in X}
\{f(y)+\alpha^{-1}D_\phi(y,x)\}\) and define
\begin{equation*}
 \Delta_\alpha(x):=
 \frac{D_\phi(x,p_\alpha(x))+D_\phi(p_\alpha(x),x)}{\alpha^2}.
\end{equation*}
If \(\nabla\phi\) is \(M\)-Lipschitz on \(X\), proximal optimality at
\(\widehat x:=p_\alpha(x)\) gives
\begin{equation*}
 \dist^2\!\left(0,\nabla f(\widehat x)+N_X(\widehat x)\right)
 \leq M\Delta_\alpha(x).
\end{equation*}
Both residuals are defined only in the interior, and this estimate requires a
uniform Lipschitz bound on \(\nabla\phi\).  Hence this framework does not
directly certify KKT stationarity at a boundary limit.

\medskip
\noindent\textbf{Spurious stationary points.}
Even when \(R_\alpha\) is extended to the boundary, its vanishing need not imply
KKT stationarity.  Take \(\overline C=\R_+^n\) and, for \(\bar x\in X\), set
\(B(\bar x):=\{i:\bar x_i=0\}\), \(I(\bar x):=\{i:\bar x_i>0\}\), and
\(M(\bar x):=\{y\in\R^n:y_i=0\text{ for }i\in B(\bar x)\}\).
For \(\phi(x)=\sum_i h(x_i)\), the boundary extension
in~\cite{chen2026spurious} freezes the zero coordinates:
\begin{equation}\label{eq:intro-boundary-step}
 \overline T_\alpha(\bar x)
 \in\arg\min_{y\in X\cap M(\bar x)}
 \left\{
  \ip{\nabla f(\bar x)}{y-\bar x}
  +\frac1\alpha\sum_{i\in I(\bar x)}D_h(y_i,\bar x_i)
 \right\}.
\end{equation}
Its residual is \(\overline R_\alpha(\bar x):=\sum_{i\in I(\bar x)}
D_h([\overline T_\alpha(\bar x)]_i,\bar x_i)\).  Its vanishing certifies only
the first-order equalities indexed by \(I(\bar x)\):
\begin{equation}\label{eq:intro-face-stationarity}
 \overline R_\alpha(\bar x)=0
 \quad\Longleftrightarrow\quad
 \exists\lambda\in\R^m:\ 
 [\nabla f(\bar x)+A^\transpose\lambda]_i=0
 \quad\bigl(i\in I(\bar x)\bigr).
\end{equation}
KKT stationarity additionally requires the boundary inequalities indexed by
\(B(\bar x)\):
\begin{equation}\label{eq:intro-active-kkt}
 \exists\lambda\in\R^m:\qquad
 [\nabla f(\bar x)+A^\transpose\lambda]_i=0
 \quad\bigl(i\in I(\bar x)\bigr),\qquad
 [\nabla f(\bar x)+A^\transpose\lambda]_j\geq0
 \quad\bigl(j\in B(\bar x)\bigr).
\end{equation}
A point satisfying \eqref{eq:intro-face-stationarity} but not
\eqref{eq:intro-active-kkt} is a spurious stationary point.  As shown
in~\cite{dingtoh2025interior}, the spurious stationary points are
precisely the non-KKT equilibria of the mirror flow~\eqref{eq:intro-flow}.  The
finite-horizon trapping and finite-exit results
in~\cite{chen2026spurious,dingtoh2025interior} leave the asymptotic question
open: they neither prove convergence to a KKT point nor provide a counterexample.
For further discussion of spurious stationary
points and their relation to mirror-flow equilibria, see
\cite{dingtoh2025interior,dingtoh2026nonkkt}.

\medskip
\noindent\textbf{Circumventing the boundary difficulty.}
Prior convergence analysis has generally avoided this difficulty through
interiority conditions, modifications of the original problem, or additional
structural assumptions.  For example, the analysis
in~\cite{mukkamala2022global} requires the entire cluster set of the iterates to
remain in the interior.  For stochastic mirror descent with random reshuffling,
Qiu et al.~\cite{qiu2026shuffling} establish last-iterate convergence under
boundedness of \(\{\nabla\phi(x_k)\}\), a condition that excludes boundary
accumulation.  Other
approaches modify the problem itself, either by adding an entropic regularizer
or a barrier term to force the relevant solution into the interior, or by explicitly
imposing additional constraints that keep the relevant variables away from
the boundary of \(\operatorname{dom}\phi\)
\cite{mukkamala2022global,scetbon2022linear}.  The boundary issue can also be
absent when problem-specific structure guarantees that local minimizers are
interior~\cite{khoo2025bregman}.  None of these results establishes
convergence of the iterates to a KKT point when the limit lies on the boundary.

\medskip
\noindent\textbf{Benign behavior under sequence convergence.}
Under the standing constraint qualification, it follows from
\cite[Proposition~5.3]{dingtoh2025interior} that every convergent sequence
generated by \eqref{eq:intro-md}, with positive nonsummable stepsizes, has a
KKT limit.  In short,
\[
 x_k\to x_\star,\qquad \alpha_k>0,\qquad
 \sum_{k=0}^\infty\alpha_k=\infty
 \quad\Longrightarrow\quad
 x_\star\ \text{satisfies the KKT condition}.
\]
Thus, boundary KKT convergence follows once sequence convergence is
established.  Several structured settings exhibit related benign behavior.
For exponentiated gradient with an Armijo line search, the objective values
converge to the optimal value for a convex differentiable objective provided
that the sequence has a strictly positive limit
point~\cite{li2019exponentiated}.  Convergence of the iterates is available for
convex mirror descent under
nonsummable but square-summable stepsizes~\cite{doan2019iterates}.  Under
variational coherence, the last iterate of stochastic mirror descent converges
almost surely to a minimizer~\cite{zhou2020convergence}.  Early continuous-time
results established trajectory convergence for Hessian--Riemannian flows under
quasiconvexity and for related singular Riemannian barrier flows with convex
objectives on polyhedra~\cite{alvarez2004hessian,attouch2004singular}.  More recent
continuous-time results establish trajectory convergence to a KKT point when
the boundary equilibria are isolated~\cite{dingtoh2025interior}.  However,
for a general nonconvex objective, sequence convergence is far from trivial:
the iterates may wind indefinitely near the boundary, as shown by the
counterexample in~\cite{dingtoh2026nonkkt}.

In nonconvex optimization, the Kurdyka--\L{}ojasiewicz (KL) inequality is a
standard tool for proving finite length and hence sequence
convergence~\cite{bolte2007lojasiewicz,attouch2013convergence}.  However, for
mirror descent, the inverse mirror metric may degenerate at the boundary, so the
same argument need not yield finite length.  To illustrate this obstruction,
consider mirror flow in the unconstrained setting \(L=\R^n\).
Then the mirror flow satisfies
\[
 \nabla^2\phi(x(t))\dot x(t)+\nabla f(x(t))=0,
 \qquad f(x(t))\downarrow f_\infty.
\]
Suppose that, on a nonstationary tail, the KL inequality
\(\vartheta'(f(x)-f_\infty)\norm{\nabla f(x)}\geq1\) holds for a concave
desingularizing function \(\vartheta\) with \(\vartheta(0)=0\) and
\(\vartheta'>0\).  When \(\nabla^2\phi=I\), mirror flow reduces to Euclidean
gradient flow, and
\[
 -\frac{\dd}{\dd t}\vartheta\bigl(f(x(t))-f_\infty\bigr)
 =\vartheta'\bigl(f(x(t))-f_\infty\bigr)\norm{\nabla f(x(t))}^2
 \geq\norm{\nabla f(x(t))}
 =\norm{\dot x(t)},
\]
so integration yields finite length.  For mirror flow, the KL inequality gives
\begin{align*}
 -\frac{\dd}{\dd t}\vartheta\bigl(f(x(t))-f_\infty\bigr)
 &=\vartheta'\bigl(f-f_\infty\bigr)
   \ip{\nabla f}{[\nabla^2\phi]^{-1}\nabla f}\geq
 \frac{\ip{\nabla f}{[\nabla^2\phi]^{-1}\nabla f}}
 {\norm{\nabla f}\norm{[\nabla^2\phi]^{-1}\nabla f}}
 \norm{\dot x}.
\end{align*}
If \(\phi\) is strongly convex and \(\nabla\phi\) is Lipschitz, then the
coefficient in the last inequality is bounded away from zero.  However, for Legendre kernels, \([\nabla^2\phi]^{-1}\) degenerates near the boundary of \(\operatorname{dom}\phi\), so this
coefficient need not be bounded away from zero and this direct KL argument
does not imply finite length.

\medskip
\noindent\textbf{Non-KKT accumulation.}
Recent work~\cite{dingtoh2026nonkkt} shows that the boundary obstruction is
genuine: mirror descent can accumulate at non-KKT boundary points.  In this
counterexample, the objective is \(C^\infty\) and entropy-relatively smooth,
and the mirror descent sequence is bounded with nonincreasing objective values
and diminishing nonsummable stepsizes.  Both the mirror-flow trajectory and
the mirror descent sequence are nonconvergent: the trajectory approaches the
boundary while winding indefinitely around a smooth closed curve of
equilibria, and the sequence lies on this trajectory and shares its
accumulation set, which contains a nonempty relatively open subset of non-KKT
points.

Motivated by these challenges, we aim to address the boundary difficulty directly.
The counterexample shows that this requires ruling out nonconvergent boundary
motion, which can produce non-KKT accumulation, whereas sequence convergence
forces a KKT limit.  The standard KL argument, however, may lose the length
control needed for sequence convergence as the inverse mirror metric
degenerates, while the usual stationarity measures need not certify the
active-constraint inequalities.  This raises the question:

\begin{quote}
\itshape
Can we establish convergence of mirror descent to a KKT point at the boundary
of \(\operatorname{dom}\phi\) under verifiable joint conditions on the
objective, the Legendre kernel, and the feasible region?
\end{quote}

\paragraph{Contributions.}
Our contributions are threefold.
\begin{enumerate}
\item \textbf{KKT convergence with boundary limits.}
We develop a metric-flattening reparameterization framework to handle the blow-up of \(\nabla^2\phi(x)\) at the boundary. Under a definable boundary
extension and a curvature condition for the kernels, we prove that
mirror descent with \(\alpha_k\in(0,1/L)\) and
\(\sum_{k=0}^{\infty}\alpha_k=\infty\) has finite length in the
reparameterized sequence and converges to a KKT point in the original sequence. We also derive
convergence rates for the reparameterized sequence from the KL exponent of the reparameterized objective. We further establish
convergence of mirror-flow trajectories to a KKT point.

\item \textbf{Uniform bound on the Lagrangian gradient.}
We introduce a general conformal-circuit argument that yields the
uniform bound
\(\|\nabla f(x)+A^\transpose\lambda\|_\infty\leq\Gamma\) on the Lagrangian
gradient, where \(\lambda\) is the affine-constraint multiplier of the mirror
subproblem. The constant \(\Gamma\) is independent of the distance from \(x\)
to the boundary of \(\operatorname{dom}\phi\). This uniform bound is a key
ingredient in the KL argument proving finite length of the reparameterized
sequence.

\item \textbf{Examples and coordinatewise rates.}
We provide several concrete examples satisfying the conditions of our
framework: Shannon entropy, Fermi--Dirac entropy, and power kernels. For these examples, transferring the KL rates from the reparameterized variables shows that different coordinates of the original sequence may converge at different rates, depending on whether their limits lie in the interior or on the boundary.
\end{enumerate}

\noindent\textbf{Independent related work.}
We compare our results with two closely related,
independently developed preprints~\cite{chen2026iterate,chen2026unified} that
appeared on arXiv shortly before our initial arXiv submission.%
\footnote{The two preprints were submitted to arXiv on August 5 and August 6,
2026, respectively; the first version of the present manuscript was submitted
to arXiv on August 7, 2026.}
In~\cite{chen2026iterate}, iterate convergence for Shannon-entropy Bregman
projected gradient under linear constraints is established through a scaled
Kurdyka--\L{}ojasiewicz property, while~\cite{chen2026unified} develops a
unified convergence framework for separable Bregman proximal point and
proximal gradient methods with composite objectives.  
%The former work~\cite{chen2026iterate} was discussed in the first
%version.  
Since the latter work \cite{chen2026unified} overlaps more
broadly with our present work, though both works are done independently, the comparison below focuses on~\cite{chen2026unified}.  The two
approaches adopt different starting points.  The iterate-convergence framework
in~\cite{chen2026unified} is formulated in the original variables and assumes
an extended scaled Kurdyka--\L{}ojasiewicz property.
The parameterization associated with the kernel is introduced afterward to
show that this assumption holds for proper subanalytic objectives that are
continuous on their domains when the kernel has a closed domain.  Our
framework instead begins with the metric-flattening reparameterization
\(z=S(x)\) itself.  We require \(S^{-1}\) to extend continuously and definably to
\(\cl S(X^\circ)\), and apply the ordinary KL property directly to the reparameterized objective
\(E\) along
\(z_k=S(x_k)\).  The reparameterization map
\(\psi(t)=\int_{t_0}^t\sqrt{\varphi''(u)}\,\dd u\)
in~\cite{chen2026unified} coincides with \(S_i\).  Thus, the two works share
the same reparameterization map as a key ingredient, but use the map at
different stages in establishing iterate convergence.  For mirror flow, the
two formulations are closer,
since~\cite{chen2026unified} also uses this map directly to obtain a Euclidean
subgradient flow.

The analysis of the required boundary-uniform estimates are also different.  The
work~\cite{chen2026unified} derives them through a detailed
analysis of the kernel and feasible-set structure, using kernel regularity
together with either polyhedrality of the feasible set or a constraint
qualification.  Our argument is more direct: a
conformal circuit decomposition gives
\(\sup_{k\geq0}\norm{\nabla f(x_k)+A^\transpose\lambda_{k+1}}_\infty<\infty\),
where \(\lambda_{k+1}\) is the subproblem multiplier.  
The results
in~\cite{chen2026unified} require stepsizes to be bounded away from zero.  Our result
allows a more flexible choice of positive nonsummable stepsizes, which may
diminish to zero.  This flexibility may facilitate potential extensions to settings
in which diminishing stepsizes are needed, such as in stochastic optimization
and subgradient methods for nonsmooth problems.  We also derive
convergence rates for the reparameterized sequence based on the KL exponent.
We then use several specific examples to illustrate the different convergence
rates of boundary and interior coordinates.
We also acknowledge that the framework in~\cite{chen2026unified} covers a broader range of problems and
algorithms, including composite
nonsmooth objectives and Bregman proximal point and proximal gradient methods (mirror descent).

\paragraph{Notation.}
For an integer \(r\geq1\), write \([r]:=\{1,\ldots,r\}\),
\(\R_+^r:=\{u\in\R^r:u_i\geq0\ \text{for all }i\in[r]\}\), and
\(\R_{++}^r:=\{u\in\R^r:u_i>0\ \text{for all }i\in[r]\}\).  The vector
\(\1_r\) denotes the all-ones vector.  For \(u,v\in\R^r\), their Hadamard product is
denoted by \(u\odot v\).  For \(v\in\R^r\), \(\Diag(v)\) denotes the diagonal
matrix with diagonal \(v\).  For a differentiable map
\(F:\R^p\to\R^q\), \(DF(x)\in\R^{q\times p}\) denotes its Jacobian at
\(x\).  Scalar functions and powers applied to vectors are
understood componentwise.  For a set \(S\), \(\cl S\) and \(\ri S\) denote its
closure and relative interior.  When \(S\) is closed and convex, \(N_S(x)\)
denotes its convex normal cone at \(x\).  For a vector \(u\),
\(\supp(u):=\{i:u_i\neq0\}\).  Unless indicated otherwise,
\(\ip{\cdot}{\cdot}\), \(\norm{\cdot}\), and \(\dist\) are the Euclidean
inner product, norm, and induced point-to-set distance.

The remainder of the paper is organized as follows.
Section~\ref{sec:preliminaries} presents the variational notions
used in the analysis.  Section~\ref{sec:geometry} develops the metric flattening
and definable boundary extension.  Section~\ref{sec:circuit}
establishes finite length in the reparameterized variables and KKT
convergence.  Section~\ref{sec:applications} applies the general framework to
Shannon entropy, Fermi--Dirac entropy, and power kernels.  Section~\ref{sec:conclusion} concludes the paper.
Some technical proofs are deferred to the appendix.

\section{\texorpdfstring{Preliminaries}{Preliminaries}}
\label{sec:preliminaries}

This section presents some preliminaries used in the convergence analysis.

\begin{definition}[Limiting subdifferential]
Let \(E:\R^d\to(-\infty,+\infty]\) be proper and lower semicontinuous, and
let \(z\in\dom E\).  Its Fr\'echet and limiting subdifferentials
\cite{rockafellar1998variational} are
\begin{align*}
 \widehat\partial E(z)
 &:=\left\{v:\liminf_{\substack{w\to z\\w\neq z}}
 \frac{E(w)-E(z)-\ip{v}{w-z}}{\norm{w-z}}\geq0\right\},\\
 \partial E(z)
 &:=\left\{v:\begin{array}{l}
 \text{there exist }z^\nu\to z\text{ and }v^\nu\to v\text{ such that}\\[-1mm]
 E(z^\nu)\to E(z)\text{ and }v^\nu\in\widehat\partial E(z^\nu)
 \end{array}\right\}.
\end{align*}
\end{definition}

\begin{definition}[Definability {\cite[Definitions~6--7]{bolte2007clarke}}]
An o-minimal structure on the real field is a sequence
\(\mathcal O=(\mathcal O_r)_{r\geq1}\), where each \(\mathcal O_r\) is a
Boolean algebra of subsets of \(\R^r\), satisfying the following properties:
\begin{enumerate}[label=\textnormal{(\roman*)}]
 \item if \(A\in\mathcal O_r\), then
       \(A\times\R\in\mathcal O_{r+1}\) and
       \(\R\times A\in\mathcal O_{r+1}\);
 \item if \(A\in\mathcal O_{r+1}\), then its projection onto the first
       \(r\) coordinates belongs to \(\mathcal O_r\);
 \item every real algebraic subset of \(\R^r\) belongs to \(\mathcal O_r\);
 \item the sets in \(\mathcal O_1\) are precisely the finite unions of
       points and intervals.
\end{enumerate}
A set \(A\subseteq\R^r\) is \emph{definable} in \(\mathcal O\) if
\(A\in\mathcal O_r\), and a function is definable if its graph is definable.
An extended-real-valued function \(E\) is definable if \(\dom E\) and the
graph of its restriction to \(\dom E\) are definable.  Throughout, all
definability conditions refer to one fixed o-minimal structure.  We write
\(\R_{\mathrm{an},\exp}\) for the o-minimal expansion of the real field by
restricted real-analytic functions and the exponential function.
\end{definition}
Definable sets are closed under finite unions, complements, and coordinate
projections.  In particular, the
image of a definable set under a definable map and the closure of
a definable set are definable.
Every proper lower-semicontinuous definable function satisfies the KL
property~\cite{kurdyka1998gradients,bolte2007lojasiewicz,bolte2007clarke}, as stated below.

\begin{definition}[Kurdyka--\L{}ojasiewicz property]
The function \(E\) has the Kurdyka--\L{}ojasiewicz (KL) property at a point
\(\bar z\) with \(\partial E(\bar z)\neq\varnothing\) if there exist a neighborhood \(U\) of \(\bar z\),
a number \(\eta>0\), and a continuous concave function
\(\vartheta:[0,\eta)\to\R_+\), with \(\vartheta(0)=0\),
\(\vartheta\in C^1(0,\eta)\), and \(\vartheta'>0\), such that
\begin{equation}\label{eq:KL}
 \vartheta'(E(z)-E(\bar z))\dist(0,\partial E(z))\geq1
\end{equation}
whenever \(z\in U\) and
\(E(\bar z)<E(z)<E(\bar z)+\eta\).  It has KL exponent
\(\theta\in[0,1)\) at \(\bar z\) if \(\vartheta\) can be chosen as
\(\vartheta(s)=c_0s^{1-\theta}\) for some \(c_0>0\).  Equivalently,
\(\dist(0,\partial E(z))\geq c(E(z)-E(\bar z))^\theta\) locally for some
\(c>0\).
\end{definition}

\begin{definition}[Relative smoothness {\cite{bauschke2017descent,lu2018relative,bolte2018first}}]
For \(L\geq0\), we say that \(f\) is \(L\)-smooth relative to \(\phi\) on
\(X^\circ\) if
\begin{equation}\label{eq:relative-smooth}
 \abs{f(y)-f(x)-\ip{\nabla f(x)}{y-x}}\leq L D_\phi(y,x)
 \qquad\forall\; x,y\in X^\circ.
\end{equation}
\end{definition}
The Legendre kernel and relative smoothness yields the following well-posedness and standard one-step descent estimate.
\begin{proposition}[Well-definedness
{\cite[Lemma~2]{bauschke2017descent}}]
\label{prop:mirror-step}
If \(X\) is compact, then for every \(x\in X^\circ\) and
\(\alpha>0\), the mirror subproblem~\eqref{eq:intro-md} has a unique
minimizer \(T_\alpha(x)\in X^\circ\).
\end{proposition}

\begin{lemma}\label{lem:bregman-decrease}
Let \(x\in X^\circ\) and \(x^+:=T_\alpha(x)\), and suppose that \(f\) is
\(L\)-smooth relative to \(\phi\).  Then
\begin{align}
 f(x)-f(x^+)
 &\geq\left(\frac1\alpha-L\right)D_\phi(x^+,x)
 +\frac1\alpha D_\phi(x,x^+).
\label{eq:bregman-decrease}
\end{align}
\end{lemma}

\begin{proof}
By the optimality condition~\eqref{eq:intro-dual-increment},
\(\ip{\nabla f(x)}{x^+-x}= {-\alpha^{-1}\ip{\nabla\phi(x^+)-\nabla\phi(x)}{x^+-x}}
=-\alpha^{-1}
\bigl(D_\phi(x^+,x)+D_\phi(x,x^+)\bigr)\).
Combining this equality with \eqref{eq:relative-smooth} yields
\eqref{eq:bregman-decrease}.
\end{proof}

\section{\texorpdfstring{Metric-Flattening Reparameterization}{Metric flattening at the boundary}}\label{sec:geometry}

This section develops the general framework of a metric-flattening
reparameterization with a definable boundary extension, which will be  used in
Section~\ref{sec:circuit} to establish convergence.
The reparameterization removes the boundary degeneracy of
\([\nabla^2\phi(x)]^{-1}\).

% \subsection{\texorpdfstring{Metric flattening and definable boundary extension}{Metric flattening and definable boundary extension}}

To handle the degeneracy of \([\nabla^2\phi(x)]^{-1}\) at the boundary, we use a
coordinatewise metric-flattening reparameterization for the separable kernel
\(\phi(x)=\sum_{i=1}^n h_i(x_i)\) for $x=(x_1,\ldots,x_n)$ on \(C=\prod_{i=1}^n I_i\).  Throughout
Sections~\ref{sec:geometry} and~\ref{sec:circuit}, we assume that each \(h_i\) is
a \(C^3\) Legendre kernel on \(I_i=(\ell_i,u_i)\) with \(h_i''>0\).
Consequently, \(h_i'(t)\to-\infty\) at a lower endpoint and
\(h_i'(t)\to+\infty\) at a upper endpoint.  We use the
lower-semicontinuous extensions of \(h_i\) and \(\phi\), with values in
\(\R\cup\{+\infty\}\).

Fix \(\bar x_i\in I_i\) and define the metric-flattening
reparameterization
\begin{equation}\label{eq:S}
 S_i(x_i):=\int_{\bar x_i}^{x_i}\sqrt{h_i''(u)}\,\dd u,\qquad
 S(x):=(S_1(x_1),\ldots,S_n(x_n)).
\end{equation}
Each \(S_i\) is a \(C^2\) increasing diffeomorphism from
\(I_i\) onto an open interval \(J_i\).  Thus \(S\) is a \(C^2\)
diffeomorphism from \(C\) onto \(S(C)\), with \(z=S(x)\) and
\(x=S^{-1}(z)\).  The Jacobian of \(S^{-1}\) at \(z\in S(C)\) is
\begin{equation}\label{eq:DSinverse}
 D(S^{-1})(z)
 =\Diag\left(
 \frac1{\sqrt{h_i''(S_i^{-1}(z_i))}}
 \right)_{i=1}^n.
\end{equation}
Consequently,
\begin{equation}\label{eq:flatten}
 D(S^{-1})(z)^\transpose\nabla^2\phi(S^{-1}(z))D(S^{-1})(z)=I_n.
\end{equation}
The reparameterized interior feasible set is
\begin{equation} \label{eq:N}
\cN^\circ:=S(X^\circ)=\{z\in S(C):A S^{-1}(z)=b\}.
\end{equation}
Its closure is denoted as \(\cN:=\cl\cN^\circ\).
By~\eqref{eq:flatten},
\(S\) transforms the Hessian metric into the identity on \(S(C)\).
The transformed metric is defined on all of \(\R^n\) and therefore remains
nondegenerate as points of \(\cN^\circ\) approach its boundary.

From a dynamical-systems perspective, the reparameterization
converts mirror flow into a subgradient flow.  Related connections between
mirror flow and reparameterized gradient flow are also studied in
\cite{li2022implicit,dingtoh2025interior}.  Suppose that \(S^{-1}\) extends continuously to
\(\cN\), and retain the same notation for this extension.  Define
the reparameterized objective \(E:\R^n\to(-\infty,+\infty]\) by
\begin{equation}\label{eq:E}
 E(z):=
 \begin{cases}
  f(S^{-1}(z)),&z\in\cN,\\
  +\infty,&z\notin\cN.
 \end{cases}
\end{equation}
For \(z\in\cN^\circ\), the sets \(\cN\) and \(\cN^\circ\) agree locally.
The full row rank of \(A D(S^{-1})(z)\) and standard subdifferential calculus
then give~\cite[Exercise~10.10]{rockafellar1998variational}
\begin{align}
 \partial E(z)
 &=\{D(S^{-1})(z)^\transpose
      (\nabla f(S^{-1}(z))+A^\transpose\mu):\mu\in\R^m\},
 \label{eq:subdiff}\\
 \dist(0,\partial E(z))
 &=\min_{\mu\in\R^m}
   \norm{D(S^{-1})(z)^\transpose
   (\nabla f(S^{-1}(z))+A^\transpose\mu)}.
 \label{eq:slope}
\end{align}
Let \(x(t)\in X^\circ\) solve \eqref{eq:intro-flow}, set
\(z(t):=S(x(t))\), and let \(\lambda(t)\) be the corresponding multiplier.
Since \(\dot x(t)=D(S^{-1})(z(t))\dot z(t)\), multiplying
\eqref{eq:intro-flow} by \(D(S^{-1})(z(t))^\transpose\) and using
\eqref{eq:flatten} give
\[
 0=\dot z(t)+D(S^{-1})(z(t))^\transpose
 \bigl(\nabla f(x(t))+A^\transpose\lambda(t)\bigr).
\]
By \eqref{eq:subdiff}, this identity yields
\begin{equation}\label{eq:subgradient-flow}
 \dot z(t)\in-\partial E(z(t)).
\end{equation}
Thus mirror flow becomes the Euclidean subgradient flow of \(E\) in the
reparameterized variables.  Under the KL property of \(E\), this equivalence
provides an intuition to proving finite length and convergence of the
reparameterized trajectory and motivates the analogous analysis of mirror
descent below.

To recover limits in the original variables, \(S^{-1}\) must extend
continuously to the boundary of \(\cN^\circ\).  The KL argument additionally
requires compactness and definability.  We formalize these conditions below.

\begin{definition}[Definable boundary extension]\label{def:boundary-extension}
Let \(S:C\to S(C)\) be the metric-flattening
reparameterization in~\eqref{eq:S}.   We say that \(S\) admits a definable
boundary extension for \(X=\overline{C}\cap L\) if the following conditions hold in
a common o-minimal structure:
\begin{enumerate}[label=(\roman*)]
 \item \(\cN:=\cl\cN^\circ\subset\R^n\) is compact and definable.
 \item The inverse restriction
       \(S^{-1}|_{\cN^\circ}:\cN^\circ\to X^\circ\) extends to a continuous
       definable map \(\cN\to X\).
\end{enumerate}
\end{definition}

%
% \begin{remark}
% Compactness of \(\cN\) is a convenient global assumption.  For a fixed
% initialization \(x_0\), it can be replaced in the convergence arguments below
% by requiring the sets
% \[
%  \mathcal L_0:=\{x\in X:f(x)\leq f(x_0)\},
%  \qquad
%  \cN_0:=\cl S(\mathcal L_0\cap X^\circ)
% \]
% to be compact, with \(\cN_0\) definable, \(S^{-1}\) extending continuously and
% definably to \(\cN_0\), and the coordinate-curvature bounds imposed on
% \(\mathcal L_0\cap X^\circ\).  Level-boundedness of \(f\) on \(X\) guarantees
% compactness of \(\mathcal L_0\), while compactness of \(\cN_0\) follows if each
% \(S_i\) is bounded on \(\{x_i:x\in\mathcal L_0\cap X^\circ\}\).  We use the
% global assumption to keep the main statements independent of \(x_0\).
% \end{remark}
%

The compactness condition is not essential to the convergence
argument and may be replaced by boundedness of the generated reparameterized
sequence, which can often be obtained from descent and level-boundedness. For convenience of presentation, we retain the global compactness assumption in the main statements. The boundary-extension condition can be verified coordinatewise as follows.

\begin{lemma}
\label{lem:boundary-extension-criterion}
For each \(i\), let \(\mathcal I_i:=\{x_i:x\in X^\circ\}\) and
\(\mathcal J_i:=S_i(\mathcal I_i)\).
Suppose that, for every \(i\), \(\cl\mathcal J_i\) is compact and
\(S_i^{-1}:\mathcal J_i\to\mathcal I_i\) extends to a continuous definable
map \(\cl\mathcal J_i\to\cl\mathcal I_i\).  Then \(S\) admits a definable
boundary extension in the sense of Definition~\ref{def:boundary-extension}.
Moreover,
\[
 \cN
 = \cl S(X^\circ)=\left\{z\in\prod_{i=1}^n\cl\mathcal J_i:A S^{-1}(z)=b\right\}.
\]
The extended inverse \(S^{-1}:\cN\to X\) is continuous, definable, and onto.
\end{lemma}

\begin{proof}
For each \(i\), the map
\(S_i^{-1}:\mathcal J_i\to\mathcal I_i\) is a strictly increasing bijection.  Density and compactness show that its continuous extension is
onto \(\cl\mathcal I_i\), while strict monotonicity gives injectivity.  It is
therefore a homeomorphism from \(\cl\mathcal J_i\) to
\(\cl\mathcal I_i\), and its inverse continuously and definably extends
\(S_i\).  Taking products yields continuous definable extensions of \(S^{-1}\)
and \(S\), for which we retain the same notation.  Since \(X^\circ=C\cap L\)
is semialgebraic and \(S\) is definable, \(S(X^\circ)\) is definable.  The
closure of a definable set is definable, so \(\cN=\cl S(X^\circ)\) is a
definable closed subset of the compact set
\(\prod_i\cl\mathcal J_i\).  To verify
Definition~\ref{def:boundary-extension}(ii), fix \(z\in\cN\) and choose
\(x^\nu\in X^\circ\) with \(S(x^\nu)\to z\).  Continuity gives
\(x^\nu\to S^{-1}(z)\), and hence \(S^{-1}(z)\in X\) because \(X\) is closed.
Thus \(S^{-1}|_{\cN^\circ}\) extends to a continuous definable map
\(\cN\to X\), so \(S\) admits a continuous definable boundary extension.

It remains to identify \(\cN\).  Set
\[
 \widetilde{\cN}:=\left\{z\in\prod_{i=1}^n\cl\mathcal J_i:
 A S^{-1}(z)=b\right\}.
\]
The inclusion \(\cN\subseteq\widetilde{\cN}\) follows from
\(S^{-1}(\cN)\subseteq X\).  Conversely, fix \(z\in\widetilde{\cN}\) and set
\(x:=S^{-1}(z)\).  Then \(x_i\in\cl\mathcal I_i\subseteq\cl I_i\) and
\(Ax=b\), so \(x\in X\).  For any \(\widehat x\in X^\circ\), the points
\(x_t:=(1-t)x+t\widehat x\) belong to \(X^\circ\) for \(t\in(0,1]\), and
continuity gives \(S(x_t)\to S(x)=z\) as $t\downarrow 0$.  Hence
\(z\in\cl S(X^\circ)=\cN\).  Applying the same
argument to any \(x\in X\) gives \(S(x)\in\cN\) and
\(S^{-1}(S(x))=x\), so the extended inverse is onto.
\end{proof}

The following proposition provides examples of definable boundary extensions.

\begin{proposition}\label{prop:kernel-boundary-extensions}
The following settings satisfy the definable boundary-extension condition.
In each case, \(h_i=h\) for every \(i\), so
\(\phi(x)=\sum_{i=1}^n h(x_i)\).
\begin{enumerate}[label=(\roman*)]
 \item Shannon entropy \(h(t)=t(\log t-1)\) for \(t\in (0,\infty)\), with the
       compact feasible set \(X=\{x\geq0:Ax=b\}\).
 \item Fermi--Dirac entropy
       \(h(t)=t\log t+(1-t)\log(1-t)\) for \(t\in (0,1)\), with
       \(X=\{x\in[0,1]^n:Ax=b\}\).
 \item For any fixed \(p\in(1,2)\), the power kernel
       \(h=h_p\), where
       \(h_p(t)=-t^{2-p}/[(p-1)(2-p)]\) for \(t\in (0,\infty)\), with the compact
       feasible set \(X=\{x\geq0:Ax=b\}\).
\end{enumerate}
\end{proposition}

\begin{proof}
\emph{(i) Shannon entropy.}
We have \(h''(t)=1/t\), so \(S'(t)=1/\sqrt{t}\) for $t\in (0,\infty)$.  Choosing
the additive constant so that the continuous extension satisfies \(S(0)=0\)
gives
\[
 {S(t)=2\sqrt{t},\; t\in [0,\infty) \qquad S^{-1}(\zeta)=\frac{\zeta^2}{4},\; \zeta\in[0,\infty).}
\]
Since \(X=\cl X^\circ\) is compact, every \(\cl\mathcal I_i\) and
\(\cl\mathcal J_i=S_i(\cl\mathcal I_i)\) is compact.  Moreover,
\(S_i^{-1}(z_i)=z_i^2/4\) extends continuously and semialgebraically to
\(\cl\mathcal J_i\).  Lemma~\ref{lem:boundary-extension-criterion} therefore
yields a definable boundary extension, with
\[
 \cN=\{z\geq0:A(z\odot z/4)=b\}.
\]

\medskip
\noindent\emph{(ii) Fermi--Dirac entropy.}
Here \(h''(t)=1/(t(1-t))\), so
\(S'(t)=1/\sqrt{t(1-t)}\) for $t\in (0,1)$.  The normalization \(S(0)=0\) gives
\[
{
 S(t)=2\arcsin\sqrt{t},\; t\in [0,1] \qquad S^{-1}(\zeta)=\sin^2(\zeta/2),\;
 \zeta\in [0,\pi].}
\]
The set \(X=\cl X^\circ\) is compact, and the continuous extension of
\(S_i\) satisfies
\(\cl\mathcal J_i=S_i(\cl\mathcal I_i)\subseteq[0,\pi]\).  Thus
\(\cl\mathcal J_i\) is compact.  The inverse
\(S_i^{-1}(z_i)=\sin^2(z_i/2)\) extends continuously and definably to this
interval.  Lemma~\ref{lem:boundary-extension-criterion} therefore yields a
definable boundary extension.  With
\(S^{-1}(z)=\sin^2(z/2)\), we similarly obtain
\[
 \cN=\{z\in[0,\pi]^n:A\sin^2(z/2)=b\}.
\]

\medskip
\noindent\emph{(iii) Power kernels.}
Fix \(p\in(1,2)\).  Then
\(h_p''(t)=t^{-p}\) for $t\in (0,\infty)$, and
\(h_p'(t)=-t^{1-p}/(p-1)\to-\infty\) as \(t\downarrow0\), so \(h_p\) is a
Legendre kernel on \((0,\infty)\).  Since
\(S_p'(t)=t^{-p/2}\) and \(2-p>0\), the normalization \(S_p(0)=0\) gives
\[
{
 S_p(t)=\frac{2}{2-p}t^{(2-p)/2},\; t\in [0,\infty) \qquad
 S_p^{-1}(\zeta)=\left(\frac{2-p}{2}\zeta\right)^{2/(2-p)},\; 
 \zeta\in [0,\infty).}
\]
Since \(X=\cl X^\circ\) is compact, every \(\cl\mathcal I_i\) is compact.
The continuous extension of \(S_p\) satisfies
\(\cl\mathcal J_i=S_p(\cl\mathcal I_i)\), so \(\cl\mathcal J_i\) is compact.
The inverse
\(S_p^{-1}(z)=((2-p)z/2)^{2/(2-p)}\) also extends continuously and definably
to zero.  Lemma~\ref{lem:boundary-extension-criterion} therefore yields a
definable boundary extension.  With
\(S_p^{-1}(z)=((2-p)z/2)^{2/(2-p)}\), we similarly obtain
\[
 \cN=\left\{z\geq0:
 A\left(\frac{2-p}{2}z\right)^{2/(2-p)}=b\right\}.
\]
All sets and maps above are definable in the common o-minimal expansion
\(\R_{\mathrm{an},\exp}\).  Hence all conditions in
Definition~\ref{def:boundary-extension} hold in each case.
\end{proof}

\section{\texorpdfstring{Reparameterized finite length and KKT convergence}{Reparameterized finite length and KKT convergence}}\label{sec:circuit}

This section establishes finite length in the reparameterized sequence and KKT convergence of mirror descent and mirror flow.  We apply the KL inequality to the reparameterized objective to prove finite length and convergence of the reparameterized sequence or trajectory.  Continuity of the extended inverse \(S^{-1}\) then recovers convergence in the original sequence.  We first present the main results.  We then derive the estimates required for the proof, using a conformal circuit decomposition to obtain the key boundary-uniform estimate for the Lagrangian gradient, and prove the theorem.

\subsection{\texorpdfstring{Main convergence results}{Main convergence results}}

Suppose that the metric-flattening reparameterization \(S\) in~\eqref{eq:S}
admits a definable boundary extension in the sense of
Definition~\ref{def:boundary-extension}, and retain the reparameterized
objective \(E\) from~\eqref{eq:E}.
The analysis relies on the following curvature bound on \((1/h_i'')'\):
\begin{equation}\label{eq:kappa}
 \kappa_i:=
 \sup_{t\in\mathcal I_i}\abs{\left(\frac1{h_i''}\right)'(t)}
 =\sup_{t\in\mathcal I_i}\frac{\abs{h_i'''(t)}}{(h_i''(t))^2},
 \qquad
 \kappa:=\max_i\kappa_i.
\end{equation}
Now, we are ready to present the main result.

\begin{theorem}[Finite length of \(S(x_k)\) and KKT convergence]
\label{thm:main}
Let each \(h_i\) be a \(C^3\) Legendre kernel on \(I_i\) with
\(h_i''>0\), and assume that \(\kappa<\infty\) and that the associated
metric-flattening reparameterization \(S\) admits a definable boundary extension.
Let \(f\) be \(C^1\) on a neighborhood of \(X\) and \(L\)-smooth relative to
\(\phi\) on \(X^\circ\) for some \(L\geq0\), with \(f|_X\) definable in the
same o-minimal structure.  Starting from \(x_0\in X^\circ\), let
\(x_{k+1}:=T_{\alpha_k}(x_k)\), where \(0<\alpha_k\leq\bar\alpha<\infty\),
\(\bar\alpha L<1\), and \(\sum_{k=0}^\infty\alpha_k=\infty\).  Then
\begin{equation}\label{eq:main-length}
 \sum_{k=0}^\infty\norm{S(x_{k+1})-S(x_k)}<\infty.
\end{equation}
Consequently, \(S(x_k)\to z_\star\in\cN\) and
\(x_k\to x_\star:=S^{-1}(z_\star)\in X\), with
\(0\in\partial E(z_\star)\).  Moreover, \(f(x_k)\) is nonincreasing and
converges to \(f(x_\star)\).  The limit \(x_\star\) is a KKT point of
\eqref{eq:problem}.
\end{theorem}

The proof is given in Section~\ref{sec:discrete-proof}.  It first derives the
one-step estimates required for the reparameterized sequence and then applies
the standard KL finite-length argument
\cite{bolte2007lojasiewicz,attouch2013convergence}.

\begin{corollary}[Convergence rates]\label{cor:rates}
Under the assumptions of Theorem~\ref{thm:main}, suppose in addition that
\(E\) has KL exponent \(\theta\in[0,1)\) at \(z_\star\).  Then there exist an
index \(k_0\) and a constant \(\gamma_0>0\) such that, with
\(\tau_k:=\sum_{j=k_0}^{k-1}\alpha_j\) and \(r_k:=f(x_k)-f(x_\star)\), the
following conclusions hold.
\begin{enumerate}[label=(\roman*)]
 \item If \(0\leq\theta<1/2\), there exists \(k_1\geq k_0\) such that
 \(x_k=x_\star\) for every \(k\geq k_1\).
 \item If \(\theta=1/2\), there is \(M>0\) such that, for all sufficiently
 large \(k\),
 \[
   r_k\leq Me^{-\gamma_0\tau_k},\qquad
   \norm{S(x_k)-z_\star}\leq Me^{-\gamma_0\tau_k/2}.
 \]
 \item If \(1/2<\theta<1\), there is \(M>0\) such that, for all sufficiently
 large \(k\),
 \[
   r_k\leq M\tau_k^{-1/(2\theta-1)},\qquad
   \norm{S(x_k)-z_\star}
   \leq M\tau_k^{-(1-\theta)/(2\theta-1)}.
 \]
\end{enumerate}
\end{corollary}

% {\color{blue}Unlike the standard proximal/PALM KL classification based on a
% next-iterate relative-error estimate, finite termination here holds throughout
% \(0\leq\theta<1/2\).  The reason is that our current-iterate relative-error
% estimate yields \(r_k-r_{k+1}\geq\gamma_0\alpha_k r_k^{2\theta}\), whereas the
% standard argument gives a recurrence involving \(r_{k+1}^{2\theta}\).}

The proof is given in Appendix~\ref{app:KL}. For mirror flow, the same reparameterization yields finite length.

\begin{theorem}\label{thm:flow}
Suppose each \(h_i\) is a \(C^3\) Legendre kernel on \(I_i\) with
\(h_i''>0\), and suppose that the associated metric-flattening
reparameterization \(S\) admits a definable boundary extension.  Let \(f\) be \(C^2\) on a neighborhood of
\(X\), and suppose that \(f|_X\) is definable in the same o-minimal structure.
For every \(x(0)\in X^\circ\), the mirror flow \eqref{eq:intro-flow} has a
unique global solution in \(X^\circ\).  Its reparameterized trajectory
\(z(t)=S(x(t))\) satisfies
\begin{equation}\label{eq:flow-length}
 \int_0^\infty\norm{\dot z(t)}\,\dd t<\infty.
\end{equation}
Consequently, \(z(t)\to z_\star\in\cN\) and
\(x(t)\to x_\star=S^{-1}(z_\star)\in X\), with
\(0\in\partial E(z_\star)\).  The limit \(x_\star\) is a KKT point of
\eqref{eq:problem}.
\end{theorem}

The proof is given in Appendix~\ref{app:flow}.

\subsection{\texorpdfstring{Proof of Theorem~\ref{thm:main}}{Proof of Theorem 4.1}}
\label{sec:discrete-proof}

We first derive boundary-uniform one-step estimates and then use
the KL property to prove finite length of the reparameterized sequence.

\subsubsection{\texorpdfstring{Conformal circuit decomposition}{Conformal circuit decomposition}}
\label{sec:conformal-circuits}

Direct bounds on the affine-constraint
multipliers or the dual sequence \(\{\nabla\phi(x_k)\}\) are generally
impossible near the boundary: the Legendre gradient blows up there, and the
multipliers may diverge as well.  We instead control the corresponding
Lagrangian gradient by decomposing the mirror-step displacement into conformal circuits
\cite{rockafellar1969elementary,muller2016elementary}, defined as follows.

\begin{definition}[A-circuit]
{A nonzero \(c\in\ker A\) is said to be} an \(A\)-\emph{circuit} if its support is
inclusion-minimal among supports of nonzero vectors in \(\ker A\)
(that is, no nonzero vector in \(\ker A\) has support strictly
contained in \(\supp(c)\)).
Two \(A\)-circuits are equivalent if they differ by a nonzero
scalar factor, and the resulting finite set of equivalence classes is denoted
by \(\cC(A)\).  For \(u,v\in\R^n\),
write \(u\sqsubseteq v\) if
\begin{equation}\label{eq:conformal}
 u_iv_i\geq0,\qquad
 u_i=0\ \text{whenever }v_i=0.
\end{equation}
\end{definition}

We use the following conformal decomposition.

\begin{lemma}[Conformal circuit decomposition
{\cite{rockafellar1969elementary,muller2016elementary}}]
\label{lem:conformal}
Every nonzero \(w\in\ker A\) can be written as
\[
 w=c^1+\cdots+c^q,\qquad
 q\leq\min\{\dim\ker A,\abs{\supp(w)}\},
\]
where each \(c^\ell\) is an \(A\)-circuit and
\(c^\ell\sqsubseteq w\).
\end{lemma}

To control the Lagrangian gradient uniformly, we introduce the following
circuit seminorm, which is invariant under shifts in
\(\operatorname{range}(A^\transpose)\).

\begin{definition}[Circuit seminorm]\label{def:circuitnorm}
If \(\ker A\neq\{0\}\), define
\begin{equation}\label{eq:circuitnorm}
 \norm{g}_{\cir(A)}
 :=\max_{[c]\in\cC(A)}
 \frac{\abs{\ip{c}{g}}}
      {\min_{j\in\supp(c)}\abs{c_j}}.
\end{equation}
If \(\ker A=\{0\}\), set
\(\norm{g}_{\cir(A)}:=0\).
\end{definition}

\begin{proposition}\label{prop:circuit-seminorm}
For \(g\in\R^n\), the quantity
\(\norm{g}_{\cir(A)}\) is well defined.  The map
\(g\mapsto\norm{g}_{\cir(A)}\) is a seminorm, and there exists
\(\chi_A<\infty\), depending only on \(A\), such that
\(\norm{g}_{\cir(A)}\leq\chi_A\norm{g}_\infty\).  Moreover,
\(\norm{g+A^\transpose\mu}_{\cir(A)}=\norm{g}_{\cir(A)}\) for every
\(\mu\in\R^m\).
\end{proposition}

\begin{proof}
The claim is immediate when \(\ker A=\{0\}\).  Suppose that
\(\ker A\neq\{0\}\).  Each circuit class is determined by its support.  Indeed,
let \(u\) and \(v\) be $A$-circuits with the same support \(J\), and fix
\(j\in J\).  If \(u\) and \(v\) are not scalar multiples, then
\(v_ju-u_jv\neq0\) and
\[
 A(v_ju-u_jv)=0,\qquad [v_ju-u_jv]_j=0.
\]
Thus \(v_ju-u_jv\) is a nonzero vector in \(\ker A\) with
\(\supp(v_ju-u_jv)\subseteq J\setminus\{j\}\), contradicting the support
minimality of \(u\).  Hence two circuits with the same support are scalar multiples and
belong to the same circuit class.  Since \([n]\) has only finitely many
subsets, \(\cC(A)\) is finite.

The quotient in \eqref{eq:circuitnorm} and the ratio below are invariant under
rescaling \(c\mapsto\tau c\), where \(\tau\neq0\).  Set
\[
 \chi_A:=
 \max_{[c]\in\cC(A)}
 \frac{\norm{c}_1}{\min_{j\in\supp(c)}\abs{c_j}}
\]
which is finite because \(\cC(A)\) is finite.  H\"older's inequality and
\eqref{eq:circuitnorm} give the bound.  Nonnegativity, absolute
homogeneity, and the triangle inequality follow from
\eqref{eq:circuitnorm} directly, so this quantity is a seminorm.  Finally,
\(Ac=0\) gives
\(\ip{c}{g+A^\transpose\mu}=\ip{c}{g}\) for every $A$-circuit \(c\), which proves
the shift invariance.
\end{proof}

The following lemma provides a uniform bound on the Lagrangian
gradient, even as the sequence approaches the boundary.

\begin{lemma}\label{lem:circuit}
For \(\phi(x)=\sum_{i=1}^n h_i(x_i)\) with \(h_i''>0\) on \(I_i\),
let \(x,x^+\in X^\circ\), \(g\in\R^n\), and \(\alpha>0\) satisfy
\[
 \nabla\phi(x^+)-\nabla\phi(x)
 +\alpha(g+A^\transpose\lambda)=0
\]
for some \(\lambda\in\R^m\).
Then
\begin{equation}\label{eq:circuit-bound}
 \norm{g+A^\transpose\lambda}_\infty\leq\norm{g}_{\cir(A)}.
\end{equation}

\end{lemma}

\begin{proof}
For every \(i\in[n]\),
\[
 \alpha\bigl(g_i+[A^\transpose\lambda]_i\bigr)
 =h_i'(x_i)-h_i'(x_i^+).
\]
Since \(h_i'\) is strictly increasing, \(x_i-x_i^+\) and
\(g_i+[A^\transpose\lambda]_i\) have the same sign and vanish
simultaneously.  If \(\ker A=\{0\}\), then \(x=x^+\) and hence
\(g+A^\transpose\lambda=0\).  Hence
\eqref{eq:circuit-bound} follows from Definition~\ref{def:circuitnorm}.

Suppose \(\ker A\neq\{0\}\).
Fix \(j\) with \(g_j+[A^\transpose\lambda]_j\neq0\), so
\(j\in\supp(x-x^+)\).  Lemma~\ref{lem:conformal} applied to \(x-x^+\)
yields an $A$-circuit \(c\sqsubseteq x-x^+\) with \(j\in\supp(c)\).  Thus
\(c_i(g_i+[A^\transpose\lambda]_i)
=\abs{c_i}\abs{g_i+[A^\transpose\lambda]_i}\) for every \(i\).
Since \(Ac=0\),
\[
 \sum_i\abs{c_i}\abs{g_i+[A^\transpose\lambda]_i}
 =\ip{c}{g+A^\transpose\lambda}
 =\ip{c}{g}.
\]
Therefore
\[
 \abs{g_j+[A^\transpose\lambda]_j}
 \leq\frac{\abs{\ip{c}{g}}}{\abs{c_j}}
 \leq\frac{\abs{\ip{c}{g}}}
          {\min_{\ell\in\supp(c)}\abs{c_\ell}}
 \leq\norm{g}_{\cir(A)}.
\]
Taking the maximum over \(j\) proves \eqref{eq:circuit-bound}.
\end{proof}

\subsubsection{\texorpdfstring{Sufficient decrease and relative-error bounds}{Sufficient decrease and relative-error bounds}}
\label{sec:sufficient-decrease-relative-error}

 Based on the preceding conformal circuit decomposition, we now establish the one-step estimates required by the finite-length argument under KL property. We first give a lower bound on the symmetric Bregman divergence in terms of the reparameterized displacement.

\begin{proposition}
\label{prop:coercivity}
For every \(x,y\in C\),
\begin{equation}\label{eq:coercivity}
 D_\phi(y,x)+D_\phi(x,y)
 \geq\norm{S(y)-S(x)}^2.
\end{equation}
\end{proposition}

\begin{proof}
For one coordinate, let
\(D_h(v,u):=h(v)-h(u)-h'(u)(v-u)\).  Then
\begin{align*}
 D_h(v,u)+D_h(u,v)
 &=\left(\int_u^v h''(t)\,\dd t\right)
   \left(\int_u^v1\,\dd t\right)\geq\left(\int_u^v\sqrt{h''(t)}\,\dd t\right)^2
 =(S(v)-S(u))^2.
\end{align*}
Summation over all the coordinates proves \eqref{eq:coercivity}.
\end{proof}

The following lemma transfers the circuit estimate to the
reparameterized variables.

\begin{lemma}\label{lem:scalar-comparison}
Fix \(i\in[n]\) with \(\kappa_i<\infty\).  Suppose
\(x_i,x_i^+\in\mathcal I_i\) and \(s_i\in\R\) satisfy
\begin{equation}\label{eq:scalar-dual}
 h_i'(x_i^+)=h_i'(x_i)-s_i.
\end{equation}
Then
\begin{equation}\label{eq:scalar-comparison}
 e^{-\kappa_i\abs{s_i}/2}
 \frac{\abs{s_i}}{\sqrt{h_i''(x_i)}}
 \leq \abs{S_i(x_i^+)-S_i(x_i)}
 \leq e^{\kappa_i\abs{s_i}/2}
 \frac{\abs{s_i}}{\sqrt{h_i''(x_i)}}.
\end{equation}
\end{lemma}

\begin{proof}
For \(u\) between \(x_i\) and \(x_i^+\), \eqref{eq:kappa} gives
\[
 \abs{\bigl(\log h_i''\bigr)'(u)}
 =\frac{\abs{h_i'''(u)}}{h_i''(u)}
 =h_i''(u)\frac{\abs{h_i'''(u)}}{(h_i''(u))^2}
 \leq\kappa_i h_i''(u).
\]
Integrating from \(x_i\) to \(u\), {and noting that $h'_i$ is increasing,} yields
\[
 \abs{\log\frac{h_i''(u)}{h_i''(x_i)}}
 \leq \kappa_i\abs{h_i'(u)-h_i'(x_i)}
 \leq \kappa_i\abs{s_i}.
\]
Hence
\[
 \frac{e^{-\kappa_i\abs{s_i}/2}}{\sqrt{h_i''(x_i)}}
 \leq\frac1{\sqrt{h_i''(u)}}
 \leq\frac{e^{\kappa_i\abs{s_i}/2}}{\sqrt{h_i''(x_i)}}.
\]
Since \(h_i''(u)>0\), multiplying the preceding bound by \(h_i''(u)\) gives
\[
 \frac{e^{-\kappa_i\abs{s_i}/2}}{\sqrt{h_i''(x_i)}}h_i''(u)
 \leq\frac{h_i''(u)}{\sqrt{h_i''(u)}}
 =\sqrt{h_i''(u)}
 \leq\frac{e^{\kappa_i\abs{s_i}/2}}{\sqrt{h_i''(x_i)}}h_i''(u).
\]
Integrating this inequality in \(u\) from \(x_i\)
to \(x_i^+\), and using \(S_i'=\sqrt{h_i''}\) and \((h_i')'=h_i''\), gives
\[
 \frac{e^{-\kappa_i\abs{s_i}/2}}{\sqrt{h_i''(x_i)}}
 \abs{h_i'(x_i^+)-h_i'(x_i)}
 \leq\abs{S_i(x_i^+)-S_i(x_i)}
 \leq\frac{e^{\kappa_i\abs{s_i}/2}}{\sqrt{h_i''(x_i)}}
 \abs{h_i'(x_i^+)-h_i'(x_i)}.
\]
Using \eqref{eq:scalar-dual} proves \eqref{eq:scalar-comparison}.
\end{proof}

\begin{proposition}
\label{prop:relative-error}
Suppose that the metric-flattening reparameterization
\(S\) in~\eqref{eq:S} admits a definable boundary extension.  Let
\(f\in C^1\) on a neighborhood of \(X\), define \(E\) by \eqref{eq:E}, and set
\(\Gamma:=\max_{x\in X}\norm{\nabla f(x)}_{\cir(A)}<\infty\).  Assume
\(\kappa<\infty\), fix \(\bar\alpha<\infty\), and let
\(x^+:=T_\alpha(x)\) for \(x\in X^\circ\) and
\(0<\alpha\leq\bar\alpha\).  With \(z:=S(x)\) and \(z^+:=S(x^+)\), we have
\begin{equation}\label{eq:relative-error}
 \dist(0,\partial E(z))
 \leq\frac{\exp(\kappa\bar\alpha\Gamma/2)}{\alpha}\norm{z^+-z}.
\end{equation}
The bound is uniform and independent of the distance to the boundary.
\end{proposition}

\begin{proof}
Note that \(z\in\cN^\circ\) and
\(x=S^{-1}(z)\).
Let \(\lambda\) be the mirror-step multiplier and write
\(r:=\nabla f(x)+A^\transpose\lambda\).  Lemma~\ref{lem:circuit} and
\eqref{eq:intro-dual-increment} give \(\abs{r_i}\leq\Gamma\) and
\(h_i'(x_i^+)=h_i'(x_i)-\alpha r_i\), respectively.
Applying Lemma~\ref{lem:scalar-comparison} with \(s_i=\alpha r_i\) gives
\[
 \frac{\abs{r_i}}{\sqrt{h_i''(x_i)}}
 \leq\frac1\alpha e^{\kappa_i\alpha\abs{r_i}/2}
       \abs{z_i^+-z_i}
 \leq\frac{\exp(\kappa\bar\alpha\Gamma/2)}{\alpha}\abs{z_i^+-z_i}.
\]
Taking \(\mu=\lambda\) in \eqref{eq:slope} and using
\eqref{eq:DSinverse} yields
\[
 \dist(0,\partial E(z))
 \leq\norm{D(S^{-1})(z)r}
 \leq\frac{\exp(\kappa\bar\alpha\Gamma/2)}{\alpha}\norm{z^+-z}.
\]
This proves \eqref{eq:relative-error}.
\end{proof}

\subsubsection{\texorpdfstring{Finite length of the reparameterized sequence}{Finite length of the reparameterized sequence}}
\label{sec:finite-length-reparameterized}

The following lemma follows the standard KL finite-length argument, with the difference that the reference point of the relative error is $z_{k}$ rather than $z_{k+1}$, and that the stepsizes may vary and need not be bounded away from zero.

\begin{lemma}
\label{lem:KL-length}
Let \(K\subset\R^d\) be compact and
\(E:\R^d\to\R\cup\{+\infty\}\) be proper,
lower-semicontinuous, continuous on \(K=\dom E\), and KL.
Suppose \(z_k\in K\), \(0<\alpha_k\leq\bar\alpha<\infty\),
\(\sum_{k=0}^\infty\alpha_k=\infty\), and there
are \(c_1,c_2>0\) such that
\begin{align}
 E(z_k)-E(z_{k+1})
 &\geq\frac{c_1}{\alpha_k}
       \norm{z_{k+1}-z_k}^2,
 \label{eq:abstract-decrease}\\
 \dist(0,\partial E(z_k))
 &\leq\frac{c_2}{\alpha_k}
       \norm{z_{k+1}-z_k}.
 \label{eq:abstract-relative}
\end{align}
Then
\(\sum_{k=0}^\infty\norm{z_{k+1}-z_k}<\infty\).
Consequently, \(z_k\) converges to a point \(z_\star\in K\) and
\(0\in\partial E(z_\star)\).
\end{lemma}

\begin{proof}
Since \(E\) is continuous on the compact set \(K\), it is bounded below.  ~\eqref{eq:abstract-decrease} gives
\(E(z_k)\downarrow E_\infty\in\R\) and
\(\sum_k\norm{z_{k+1}-z_k}^2/\alpha_k<\infty\).  Since
\(\alpha_k\leq\bar\alpha\), we also have
\(\sum_k\norm{z_{k+1}-z_k}^2<\infty\) and
\(\norm{z_{k+1}-z_k}\to0\).  Combining the two estimates gives
\begin{equation}\label{eq:slope-descent}
 E(z_k)-E(z_{k+1})
 \geq\frac{c_1}{c_2^2}\alpha_k
 \dist(0,\partial E(z_k))^2.
\end{equation}
Thus
\(\sum_k\alpha_k\dist(0,\partial E(z_k))^2<\infty\).
Since \(\sum_k\alpha_k=\infty\), this implies
\(\liminf_k\dist(0,\partial E(z_k))=0\).  Thus there is a
subsequence \(\{z_{k_j}\}\) in $K$ such that \(z_{k_j}\to z_\star\in K\) and
\(\dist(0,\partial E(z_{k_j}))\to0\).
~\eqref{eq:abstract-relative} ensures that
\(\partial E(z_k)\neq\varnothing\) for every \(k\).  Since each limiting
subdifferential is closed in finite dimensions, choose
\(v_{k_j}\in\partial E(z_{k_j})\) with
\[
 \norm{v_{k_j}}=\dist(0,\partial E(z_{k_j})).
\]
Thus \(z_{k_j}\to z_\star\), \(v_{k_j}\to0\), and continuity on \(K\)
gives \(E(z_{k_j})\to E(z_\star)=E_\infty\).  Sequential closedness of
the limiting subdifferential yields \(0\in\partial E(z_\star)\).

If \(E(z_k)=E_\infty\) for some \(k\), monotonicity gives
\(E(z_\ell)=E_\infty\) for every \(\ell\ge k\).  Equation~\eqref{eq:abstract-decrease}
then gives \(z_{\ell+1}=z_\ell\), and \eqref{eq:abstract-relative} gives
\(0\in\partial E(z_k)\).  Thus the conclusion of the lemma already holds.
Now assume \(E(z_k)>E_\infty\) for every \(k\). Apply the KL inequality at \(z_\star\), with neighborhood \(U\), width
\(\eta\), and desingularizing function \(\vartheta\).  Whenever
\(z_k\in U\) and \(0<E(z_k)-E_\infty<\eta\), we have
\begin{equation}\label{eq:KL-slope-step}
 1
 \leq\vartheta'(E(z_k)-E_\infty)\dist(0,\partial E(z_k))
 \leq\frac{c_2}{\alpha_k}
       \vartheta'(E(z_k)-E_\infty)\norm{z_{k+1}-z_k}.
\end{equation}
Concavity of \(\vartheta\), \eqref{eq:abstract-decrease}, and
\eqref{eq:KL-slope-step} then give
\begin{align}
 \vartheta(E(z_k)-E_\infty)-\vartheta(E(z_{k+1})-E_\infty)
 &\geq\vartheta'(E(z_k)-E_\infty)
       (E(z_k)-E(z_{k+1}))\notag\\
 &\geq\frac{c_1}{\alpha_k}
       \vartheta'(E(z_k)-E_\infty)\norm{z_{k+1}-z_k}^2
 \geq\frac{c_1}{c_2}
       \norm{z_{k+1}-z_k}.\label{eq:KL-tail-descent}
\end{align}
Choose \(\rho>0\) such that
\(\mathbb B(z_\star,2\rho)\cap K\subset U\).  For sufficiently large \(j\),
\[
 \norm{z_{k_j}-z_\star}<\rho,\qquad
 0<E(z_{k_j})-E_\infty<\eta,\qquad
 \frac{c_2}{c_1}
 \vartheta(E(z_{k_j})-E_\infty)<\rho.
\]
Suppose that \(N>k_j\) is the first index for which
\(z_N\notin\mathbb B(z_\star,2\rho)\).  Monotonicity gives
\(E(z_k)-E_\infty<\eta\) for \(k=k_j,\ldots,N-1\), so \eqref{eq:KL-tail-descent} is valid for
\(k=k_j,\ldots,N-1\).  Therefore
\(\sum_{k=k_j}^{N-1}\norm{z_{k+1}-z_k}
\leq(c_2/c_1)\vartheta(E(z_{k_j})-E_\infty)<\rho\), which implies
\(\norm{z_N-z_\star}\leq\norm{z_{k_j}-z_\star}
+\sum_{k=k_j}^{N-1}\norm{z_{k+1}-z_k}<2\rho\), and we get a contradiction.
Thus the tail remains in the KL neighborhood.
For every \(k\geq k_j\), summing \eqref{eq:KL-tail-descent} gives
\[
 \sum_{\ell=k}^{\infty}\norm{z_{\ell+1}-z_\ell}
 \leq\frac{c_2}{c_1}
       \vartheta(E(z_k)-E_\infty).
\]
Thus \(\sum_{k=0}^\infty\norm{z_{k+1}-z_k}<\infty\), so \(z_k\) is
Cauchy.  Since
\(z_\star\) is a cluster point, the sequence converges to \(z_\star\).
\end{proof}

We are now ready to prove Theorem~\ref{thm:main}.

\begin{proof}[Proof of Theorem~\ref{thm:main}]
The boundary-extension condition makes
\(X=S^{-1}(\cN)\) compact.  Proposition~\ref{prop:mirror-step} therefore
ensures that every mirror step is well defined and belongs to \(X^\circ\).
For \(z_k:=S(x_k)\),
\(z_k\in\cN=\dom E\) and \(E(z_k)=f(x_k)\). By definition, \(E\) is proper and lower semicontinuous, continuous on
\(\cN\), and KL by definability.  For every \(k\),
Lemma~\ref{lem:bregman-decrease} and Proposition~\ref{prop:coercivity} give
\[
 E(z_k)-E(z_{k+1})
 \geq\frac{1-\alpha_kL}{\alpha_k}
       \norm{z_{k+1}-z_k}^2
 \geq\frac{1-\bar\alpha L}{\alpha_k}
       \norm{z_{k+1}-z_k}^2.
\]
Thus \eqref{eq:abstract-decrease} holds with \(c_1=1-\bar\alpha L\), while
Proposition~\ref{prop:relative-error} gives \eqref{eq:abstract-relative} with
\(c_2=\exp(\kappa\bar\alpha\Gamma/2)\).
Lemma~\ref{lem:KL-length} then gives \eqref{eq:main-length},
\(z_k\to z_\star\), and \(0\in\partial E(z_\star)\).
Continuity of \(S^{-1}\) and \(E|_{\cN}\) gives
\(x_k\to x_\star\) and \(f(x_k)\to f(x_\star)\).  Moreover,
\(E(z_k)-E(z_{k+1})\geq0\) for \(\bar\alpha L<1\), so
\(f(x_{k+1})\leq f(x_k)\) for every \(k\).  Since \(\{x_k\}\) is convergent,
\cite[Proposition~5.3]{dingtoh2025interior} implies that \(x_\star\) is a KKT
point of \eqref{eq:problem}.
\end{proof}

\section{\texorpdfstring{Applications}{Applications}}
\label{sec:applications}

In this section, we apply Theorem~\ref{thm:main} to mirror descent with
Shannon entropy, Fermi--Dirac entropy, and power kernels, and use
Corollary~\ref{cor:rates} to derive rates for the original sequence
\(\{x_k\}\).

\subsection{Shannon entropy}

Consider \(C:=\R_{++}^n\) and \(X:=\{x\in\R_+^n:Ax=b\}\), with
\(X^\circ=X\cap\R_{++}^n\neq\varnothing\), and assume that \(X\) is
compact.  The Shannon entropy kernel is
\(\phi(x):=\sum_{i=1}^n x_i(\log x_i-1)\) for \(x\in\R_{++}^n\).
The metric-flattening reparameterization and its inverse are
\(S(x)=2\sqrt{x}\) and \(S^{-1}(z)=z\odot z/4\), where the square root is applied coordinatewise.

\begin{corollary}
\label{cor:entropy-affine}
Let \(f\) be \(C^1\) on a neighborhood of \(X\), suppose that \(f|_X\) is
definable and \(f\) is \(L\)-smooth relative to \(\phi\) on
\(X^\circ\).  Starting from \(x_0\in X^\circ\), let
\(x_{k+1}:=T_{\alpha_k}(x_k)\), where
\(0<\alpha_k\leq\bar\alpha<\infty\), \(\bar\alpha L<1\), and
\(\sum_{k=0}^\infty\alpha_k=\infty\).  Then
\(\sum_{k=0}^\infty
\norm{2\sqrt{x_{k+1}}-2\sqrt{x_k}}<\infty\), and \(x_k\) converges to
a KKT point \(x_\star\) of \(f\) on \(X\).
\end{corollary}

\begin{proof}
Proposition~\ref{prop:kernel-boundary-extensions} gives the
semialgebraic boundary extension
\[
 \cN=\cl S(X^\circ)
 =\{z\in\R_+^n:A(z\odot z/4)=b\}.
\]
For \(h(t)=t(\log t-1)\) with $t\in (0,\infty)$,
\(\abs{h'''(t)}/(h''(t))^2=1\), so \(\kappa=1\).
Compactness of \(X\) makes the circuit quantity \(\Gamma\) finite.
Theorem~\ref{thm:main} then yields
finite length and KKT convergence.
\end{proof}

\begin{corollary}
\label{cor:entropy-affine-rates}
Under the assumptions of Corollary~\ref{cor:entropy-affine}, suppose that
\(E\) has KL exponent \(\theta\in[0,1)\) at the limit
\(z_\star=2\sqrt{x_\star}\).  Let
\(I_0:=\{i:x_{\star,i}=0\}\) and \(I_+:=[n]\setminus I_0\).
If \(0\leq\theta<1/2\), then \(x_k=x_\star\) for all sufficiently large
\(k\).  If \(1/2\leq\theta<1\), take \(k_0\) and \(\gamma_0\) as in
Corollary~\ref{cor:rates}, let
\(\tau_k:=\sum_{j=k_0}^{k-1}\alpha_j\), and define
\[
 \rho_k:=
 \begin{cases}
  \exp(-\gamma_0\tau_k/2),&\theta=1/2,\\[5pt]
  \tau_k^{-(1-\theta)/(2\theta-1)},&1/2<\theta<1,
 \end{cases}
\]
Then
\[
 \norm{(x_k-x_\star)_{I_+}}=O(\rho_k),
 \qquad
 \norm{(x_k)_{I_0}}=O(\rho_k^2),
 \qquad
 \norm{x_k-x_\star}=O(\rho_k).
\]
\end{corollary}

\begin{proof}
Corollary~\ref{cor:rates} gives
\(\norm{z_k-z_\star}=O(\rho_k)\).  Note that
\(x_{k,i}-x_{\star,i}
=z_{\star,i}(z_{k,i}-z_{\star,i})/2
+(z_{k,i}-z_{\star,i})^2/4\).
If \(i\in I_+\), then \(z_{\star,i}>0\) and
\(\abs{x_{k,i}-x_{\star,i}}=O(\rho_k)\).  If \(i\in I_0\), then
\(z_{\star,i}=0\) and
\(x_{k,i}=(z_{k,i}-z_{\star,i})^2/4=O(\rho_k^2)\).
Since both coordinate sets are finite, these estimates give the claimed
bounds.
\end{proof}

The coordinatewise rates for the original iterates can differ according to
whether a coordinate converges to the boundary or remains in the interior.
In this case, a coordinate converging to the boundary has the
faster \(O(\rho_k^2)\) rate, whereas a coordinate converging to a positive
limit has rate \(O(\rho_k)\).
When \(\alpha_k\equiv\alpha>0\) and \(1/2<\theta<1\), these rates become
\(O(k^{-2(1-\theta)/(2\theta-1)})\) at the boundary and
\(O(k^{-(1-\theta)/(2\theta-1)})\) in the interior.

\subsection{Fermi--Dirac entropy}

Consider \(C:=(0,1)^n\) and
\(X:=\{x\in[0,1]^n:Ax=b\}\), with
\(X^\circ=X\cap(0,1)^n\neq\varnothing\).  The Fermi--Dirac entropy kernel is
\(\phi(x):=\sum_{i=1}^n
[x_i\log x_i+(1-x_i)\log(1-x_i)]\) for \(x\in(0,1)^n\).
Its metric-flattening reparameterization and inverse are
\(S(x)=2\arcsin\sqrt{x}\) and \(S^{-1}(z)=\sin^2(z/2)\), respectively,
with all scalar operations applied coordinatewise.

\begin{corollary}
\label{cor:fermi-affine}
Let \(f\) be \(C^1\) on a neighborhood of \(X\), suppose that \(f|_X\) is
definable and  \(f\) is \(L\)-smooth relative to \(\phi\) on
\(X^\circ\).  Starting from \(x_0\in X^\circ\), let
\(x_{k+1}:=T_{\alpha_k}(x_k)\), where
\(0<\alpha_k\leq\bar\alpha<\infty\), \(\bar\alpha L<1\), and
\(\sum_{k=0}^\infty\alpha_k=\infty\).  Then
\(\sum_{k=0}^\infty\norm{S(x_{k+1})-S(x_k)}<\infty\), and \(x_k\)
converges to a KKT point \(x_\star\) of \(f\) on \(X\).
\end{corollary}

\begin{proof}
Proposition~\ref{prop:kernel-boundary-extensions} gives the definable
boundary extension
\[
 \cN=\cl S(X^\circ)
 =\{z\in[0,\pi]^n:A\sin^2(z/2)=b\}.
\]
For the Fermi--Dirac kernel,
\(\abs{h'''(t)}/(h''(t))^2=\abs{1-2t}\leq1\) for $t\in(0,1)$, so $\kappa = 1$.  Compactness of \(X\)
makes the circuit quantity \(\Gamma\) finite.
The conclusion follows from Theorem~\ref{thm:main}.
\end{proof}

As in the Shannon-entropy case, the coordinatewise rates for the original
iterates differ according to whether the limiting coordinate lies on the
boundary or in the interior.

\begin{corollary}
\label{cor:fermi-affine-rates}
Under the assumptions of Corollary~\ref{cor:fermi-affine}, suppose that
\(E\) has KL exponent \(\theta\in[0,1)\) at the limit
\(z_\star=S(x_\star)\).  Let
\(I_0:=\{i:x_{\star,i}=0\}\), \(I_1:=\{i:x_{\star,i}=1\}\), and
\(I_{\mathrm{int}}:=[n]\setminus(I_0\cup I_1)\).  If
\(0\leq\theta<1/2\), then \(x_k=x_\star\) for all sufficiently large
\(k\).  If \(1/2\leq\theta<1\), take \(k_0\) and \(\gamma_0\) as in
Corollary~\ref{cor:rates}, let \(\tau_k:=\sum_{j=k_0}^{k-1}\alpha_j\), and define
\[
 \rho_k:=
 \begin{cases}
  \exp(-\gamma_0\tau_k/2),&\theta=1/2,\\[5pt]
  \tau_k^{-(1-\theta)/(2\theta-1)},&1/2<\theta<1,
 \end{cases}
\]
Then
\[
 \norm{(x_k-x_\star)_{I_{\mathrm{int}}}}=O(\rho_k),\qquad
 \norm{(x_k)_{I_0}}=O(\rho_k^2),\qquad
 \norm{(\1_n-x_k)_{I_1}}=O(\rho_k^2),
\]
and \(\norm{x_k-x_\star}=O(\rho_k)\).
\end{corollary}

\begin{proof}
Corollary~\ref{cor:rates} gives
\(\norm{S(x_k)-z_\star}=O(\rho_k)\).  For each \(i\in[n]\),
\[
 x_{k,i}-x_{\star,i}
 =\sin\!\left(\frac{z_{k,i}+z_{\star,i}}{2}\right)
  \sin\!\left(\frac{z_{k,i}-z_{\star,i}}{2}\right).
\]
If \(i\in I_{\mathrm{int}}\), this identity gives
\(\abs{x_{k,i}-x_{\star,i}}=O(\rho_k)\).  If \(i\in I_0\), then
\(z_{\star,i}=0\) and
\(x_{k,i}=\sin^2((z_{k,i}-z_{\star,i})/2)=O(\rho_k^2)\).  If
\(i\in I_1\), then \(z_{\star,i}=\pi\) and
\(1-x_{k,i}=\sin^2((z_{k,i}-z_{\star,i})/2)=O(\rho_k^2)\).
The coordinatewise estimates give the final bound.
\end{proof}

\subsection{Power kernels}

Consider \(C:=\R_{++}^n\) and \(X:=\{x\in\R_+^n:Ax=b\}\), with
\(X^\circ=X\cap\R_{++}^n\neq\varnothing\), and assume that \(X\) is
compact.  For \(1\leq p<2\), define \(h_p\), up to an affine term, by
\(h_p''(t)=t^{-p}\) for \(t\in (0,\infty)\), with the closed extension normalized
by \(h_p(0)=0\) and \(h_p(t)=+\infty\) for \(t<0\).  Since \(h_p''>0\) and
\(h_p'(t)\to-\infty\) as \(t\downarrow0\), \(h_p\) is a Legendre kernel on
\((0,\infty)\).  The metric-flattening reparameterization and its inverse are
\(S_p(t)=t^{1-p/2}/(1-p/2)\) for $t\in [0,\infty)$ and
\(S_p^{-1}(\zeta)=((1-p/2)\zeta)^{1/(1-p/2)}\) for $\zeta\in[0,\infty)$.  Moreover,
\(\abs{h_p'''(t)}/(h_p''(t))^2=pt^{p-1}\).
Because \(p<2\), both maps extend continuously to zero and are definable in
\(\R_{\mathrm{an},\exp}\).

For each \(i\in[n]\), fix
\(p_i\in[1,2)\), and set
\(\phi(x):=\sum_{i=1}^n h_{p_i}(x_i)\) and
\(S(x):=(S_{p_i}(x_i))_{i=1}^n\).

\begin{corollary}
\label{cor:power}
Let \(f\) be \(C^1\) on a neighborhood of \(X\), suppose that
\(f|_X\) is definable in \(\R_{\mathrm{an},\exp}\) and 
\(f\) is \(L\)-smooth relative to \(\phi\) on \(X^\circ\).  Starting
from \(x_0\in X^\circ\), let
\(x_{k+1}:=T_{\alpha_k}(x_k)\), where
\(0<\alpha_k\leq\bar\alpha<\infty\), \(\bar\alpha L<1\), and
\(\sum_{k=0}^\infty\alpha_k=\infty\).  Then
\(\sum_{k=0}^\infty\norm{S(x_{k+1})-S(x_k)}<\infty\), and \(x_k\)
converges to a KKT point of \(f\) on \(X\).
\end{corollary}

\begin{proof}
The coordinate inverses above and Lemma~\ref{lem:boundary-extension-criterion}
give the definable
boundary extension
\[
 \cN=\cl S(X^\circ)=\{z\geq0:AS^{-1}(z)=b\}.
\]
For \(h_{p_i}\),
\(\abs{h_{p_i}'''(t)}/(h_{p_i}''(t))^2=p_i t^{p_i-1}\) for $t\in (0,\infty)$, so compactness of
\(X\) gives \(\kappa<\infty\).  Compactness of \(X\) also makes the circuit
quantity \(\Gamma\) finite.
Theorem~\ref{thm:main} yields finite length and KKT convergence.
\end{proof}

\begin{corollary}
\label{cor:power-rates}
Under the assumptions of Corollary~\ref{cor:power}, suppose that
\(E\) has KL exponent \(\theta\in[0,1)\) at the limit
\(z_\star=S(x_\star)\).  Let
\(I_0:=\{i:x_{\star,i}=0\}\), \(I_+:=[n]\setminus I_0\), and
\(q_i:=2/(2-p_i)\geq2\).  If \(0\leq\theta<1/2\), then
\(x_k=x_\star\) for all sufficiently large \(k\).  If
\(1/2\leq\theta<1\), take \(k_0\) and \(\gamma_0\) as in
Corollary~\ref{cor:rates}, let \(\tau_k:=\sum_{j=k_0}^{k-1}\alpha_j\), and define
\[
 \rho_k:=
 \begin{cases}
  \exp(-\gamma_0\tau_k/2),&\theta=1/2,\\[5pt]
  \tau_k^{-(1-\theta)/(2\theta-1)},&1/2<\theta<1,
 \end{cases}
\]
Then
\[
 \norm{(x_k-x_\star)_{I_+}}=O(\rho_k),
 \qquad
 x_{k,i}=O(\rho_k^{q_i})\quad(i\in I_0),
 \qquad
 \norm{x_k-x_\star}=O(\rho_k).
\]
\end{corollary}

\begin{proof}
Corollary~\ref{cor:rates} gives
\(\norm{S(x_k)-z_\star}=O(\rho_k)\).  For each \(i\in[n]\),
\[
 x_{k,i}-x_{\star,i}
 =\bigl((1-p_i/2)S_{p_i}(x_{k,i})\bigr)^{q_i}
  -\bigl((1-p_i/2)z_{\star,i}\bigr)^{q_i}.
\]
If \(i\in I_+\), then \(z_{\star,i}>0\), and the mean-value theorem gives
\(\abs{x_{k,i}-x_{\star,i}}=O(\rho_k)\).  If \(i\in I_0\), then
\(z_{\star,i}=0\) and
\(x_{k,i}=((1-p_i/2)(S_{p_i}(x_{k,i})-z_{\star,i}))^{q_i}
=O(\rho_k^{q_i})\).  This proves the bounds.
\end{proof}

\section{Conclusion}\label{sec:conclusion}

We prove that mirror descent converges to a KKT point for the nonconvex problem
without excluding boundary limits {under suitable assumptions.}  The result holds under verifiable
conditions that jointly couple the objective, the Legendre kernel, and the
feasible region.  The key ingredient to establish the convergence is a metric-flattening
reparameterization \(S\) that admits a definable boundary extension.
Applying the KL argument to the reparameterized objective yields convergence
of \(S(x_k)\).  Continuity of \(S^{-1}\) then recovers convergence to the KKT point of the
original sequence.  We further apply our general framework to some concrete examples: Shannon entropy,
Fermi--Dirac entropy, and power kernels.  Future work may include more general
constraints and nonseparable kernels, mirror descent with
stochastic gradients or subgradients, and broader Bregman-type methods, such as
Bregman alternating minimization and Bregman ADMM, including their inexact variants.

\medskip
\noindent\textbf{Acknowledgments.} The authors thank Gesualdo Scutari for
helpful comments. Over the course of this project, the authors used multiple
versions of OpenAI's ChatGPT, beginning with GPT-4o, for brainstorming,
exploratory discussions, and language editing.

\appendix

\section{Proof of Corollary~\ref{cor:rates}}\label{app:KL}

Along the mirror descent sequence,
\(r_k=E(z_k)-E(z_\star)\).
The KL-exponent condition and \eqref{eq:slope-descent} give an index \(k_0\)
and a constant \(\gamma_0>0\) such that
\begin{equation}\label{eq:rate-recurrence}
 r_k-r_{k+1}\geq \gamma_0\alpha_kr_k^{2\theta}
 \qquad(k\geq k_0).
\end{equation}
If \(r_k=0\) at some index, \eqref{eq:abstract-decrease} gives
\(z_\ell=z_k=z_\star\) and \(r_\ell=0\) for every \(\ell\geq k\).
Now assume \(r_k>0\) on the tail under consideration.
If \(0\leq\theta<1/2\), then
\[
 r_k^{1-2\theta}-r_{k+1}^{1-2\theta}
 =(1-2\theta)\int_{r_{k+1}}^{r_k}t^{-2\theta}\,\dd t
 \geq(1-2\theta)r_k^{-2\theta}(r_k-r_{k+1})
 \geq (1-2\theta)\gamma_0\alpha_k.
\]
If \(r_k\) remained positive, summation would contradict
\(\sum_k\alpha_k=\infty\).  Hence \(r_{k_1}=0\) for some
\(k_1\geq k_0\).  Since \(E(z_k)\downarrow E(z_\star)\),
~\eqref{eq:abstract-decrease} gives \(z_k=z_{k_1}\) for every
\(k\geq k_1\).  Passing to the limit yields \(z_{k_1}=z_\star\), and
therefore \(x_k=x_\star\) for every \(k\geq k_1\).

If \(\theta=1/2\), then
\[
 r_{k+1}\leq(1-\gamma_0\alpha_k)r_k.
\]
Since \(r_{k+1}>0\), this inequality gives
\(0<1-\gamma_0\alpha_k<1\).  Hence
\[
 r_k
 \leq r_{k_0}\prod_{j=k_0}^{k-1}(1-\gamma_0\alpha_j)
 \leq r_{k_0}\exp(-\gamma_0\tau_k).
\]

If \(1/2<\theta<1\), then \eqref{eq:rate-recurrence} gives
\[
 r_{k+1}^{-(2\theta-1)}-r_k^{-(2\theta-1)}
 =(2\theta-1)\int_{r_{k+1}}^{r_k}t^{-2\theta}\,\dd t
 \geq (2\theta-1)r_k^{-2\theta}(r_k-r_{k+1})
 \geq (2\theta-1)\gamma_0\alpha_k.
\]
Summation yields
\[
 r_k\leq
 \left(
 r_{k_0}^{-(2\theta-1)}
 +(2\theta-1)\gamma_0\tau_k
 \right)^{-1/(2\theta-1)}.
\]
The tail estimate in the proof of
Lemma~\ref{lem:KL-length} gives, for all sufficiently large \(k\),
\[
 \norm{z_k-z_\star}
 \leq\sum_{j\geq k}\norm{z_{j+1}-z_j}
 \leq M\vartheta(r_k)
 =Mr_k^{1-\theta},
\]
and substituting the preceding energy bounds yields the corresponding rates
for \(\norm{S(x_k)-z_\star}\) in Corollary~\ref{cor:rates}.

\section{Proof of Theorem~\ref{thm:flow}}\label{app:flow}

Fix \(x(0)\in X^\circ\).  By \cite[Proposition~3.1]{dingtoh2025interior},
\eqref{eq:intro-flow} admits a global solution that remains in \(X^\circ\).
Since the induced vector field is locally \(C^1\), this solution is unique.
For \(z(t)=S(x(t))\), \eqref{eq:subgradient-flow} gives
\(-\dot z(t)\in\partial E(z(t))\).
To show that this is the minimum-norm element, let \(\lambda(t)\) be the
multiplier in \eqref{eq:intro-flow}.  Since
\(\dot x(t)=D(S^{-1})(z(t))\dot z(t)\), the Hessian form of
\eqref{eq:intro-flow} and \eqref{eq:flatten} give
\[
\begin{aligned}
0
&=D(S^{-1})(z(t))^\transpose
  \bigl(\nabla^2\phi(x(t))\dot x(t)
        +\nabla f(x(t))+A^\transpose\lambda(t)\bigr)\\
&=D(S^{-1})(z(t))^\transpose\,
  \nabla^2\phi(S^{-1}(z(t)))\,D(S^{-1})(z(t))\dot z(t)
  +D(S^{-1})(z(t))^\transpose
   \bigl(\nabla f(x(t))+A^\transpose\lambda(t)\bigr)\\
&=\dot z(t)+D(S^{-1})(z(t))^\transpose
  \bigl(\nabla f(x(t))+A^\transpose\lambda(t)\bigr).
\end{aligned}
\]
Therefore,
\[
 -\dot z(t)
 =D(S^{-1})(z(t))^\transpose
  \bigl(\nabla f(x(t))+A^\transpose\lambda(t)\bigr).
\]
Now fix \(v\in\partial E(z(t))\).  By \eqref{eq:subdiff}, there exists
\(\eta\in\R^m\) such that
\[
 v=D(S^{-1})(z(t))^\transpose
   \bigl(\nabla f(x(t))+A^\transpose\eta\bigr).
\]
Subtracting the two equalities and setting
\(\mu=\eta-\lambda(t)\) yield
\[
 v+\dot z(t)=D(S^{-1})(z(t))^\transpose A^\transpose\mu,
 \qquad
 v=-\dot z(t)+D(S^{-1})(z(t))^\transpose A^\transpose\mu.
\]
Differentiating \(A S^{-1}(z(t))=b\) gives {\(A D(S^{-1})(z(t))\dot z(t)=0\) and hence}
\[
 \ip{\dot z(t)}{v+\dot z(t)}
 =\ip{A D(S^{-1})(z(t))\dot z(t)}{\mu}=0.
\]
Consequently,
\[
 \norm{v}^2
 =\norm{\dot z(t)}^2+\norm{v+\dot z(t)}^2
\]
Thus \(-\dot z(t)\) is the minimum-norm element of
\(\partial E(z(t))\).  By \eqref{eq:E},
\(\nabla(f\circ S^{-1})(z(t))\in\partial E(z(t))\).  Hence
\begin{align}\label{eq:flow-identities}
 \dist(0,\partial E(z(t)))&=\norm{\dot z(t)},
 \\
 \frac{\dd}{\dd t}E(z(t))
 &=\ip{\nabla(f\circ S^{-1})(z(t))}{\dot z(t)}
 =-\norm{\dot z(t)}^2.
\end{align}
Thus \(E(z(t))\downarrow E_\infty\in\R\) and
\(\int_0^\infty\norm{\dot z(t)}^2\,\dd t<\infty\).  Hence there exist
times \(t_j\to\infty\) such that \(\dot z(t_j)\to0\).  By compactness of
\(\cN\), after passing to a subsequence, \(z(t_j)\to z_\star\in\cN\).
Continuity of \(E\) on \(\cN\) and closedness of \(\partial E\) yield
\(E(z_\star)=E_\infty\) and \(0\in\partial E(z_\star)\).

If \(E(z(t_0))=E_\infty\) for some \(t_0<\infty\), monotonicity gives
\(E(z(t))=E_\infty\) for every \(t\geq t_0\), and
\eqref{eq:flow-identities} implies \(\dot z(t)=0\) on
\([t_0,\infty)\).  Now consider the case \(E(z(t))>E_\infty\) for every
\(t\geq0\).  Let \(U\), \(\eta\), and \(\vartheta\) be the neighborhood, width, and
desingularizing function in the KL inequality at \(z_\star\).  Whenever
\(z(t)\in U\) and \(0<E(z(t))-E_\infty<\eta\),
\eqref{eq:KL} and \eqref{eq:flow-identities} give
\begin{equation}\label{eq:flow-KL-length}
 -\frac{\dd}{\dd t}\vartheta(E(z(t))-E_\infty)
 =\vartheta'(E(z(t))-E_\infty)\norm{\dot z(t)}^2
 \geq\norm{\dot z(t)}.
\end{equation}
Choose \(\rho>0\) such that
\(\{z:\norm{z-z_\star}\leq2\rho\}\subset U\).  Since
\(z(t_j)\to z_\star\), \(E(z(t_j))\to E_\infty\), and
\(\vartheta(s)\to0\) as \(s\downarrow0\), fix \(j\) such that
\[
 \norm{z(t_j)-z_\star}<\rho,\qquad
 0<E(z(t_j))-E_\infty<\eta,\qquad
 \vartheta(E(z(t_j))-E_\infty)<\rho.
\]
Define \(T:=\inf\{t\geq t_j:\norm{z(t)-z_\star}\geq2\rho\}\), with
\(T=\infty\) if the set is empty.  For every
\(s\in[t_j,T)\), monotonicity and \eqref{eq:flow-KL-length} give
\[
 \int_{t_j}^{s}\norm{\dot z(t)}\,\dd t
 \leq
 \vartheta(E(z(t_j))-E_\infty)
 -\vartheta(E(z(s))-E_\infty)
 <\rho.
\]
Consequently,
\[
 \norm{z(s)-z_\star}
 \leq\norm{z(t_j)-z_\star}
     +\int_{t_j}^{s}\norm{\dot z(t)}\,\dd t
 <2\rho.
\]
If \(T<\infty\), letting \(s\uparrow T\) contradicts the definition of
\(T\).  Therefore \(T=\infty\).  Letting \(s\to\infty\) in the integral
estimate gives
\[
 \int_{t_j}^{\infty}\norm{\dot z(t)}\,\dd t
 \leq\vartheta(E(z(t_j))-E_\infty)<\infty,
\]
which proves \eqref{eq:flow-length}.  Hence \(z(t)\) converges.  Since
\(z(t_j)\to z_\star\), its limit is \(z_\star\).  Continuity of \(S^{-1}\)
then gives \(x(t)\to x_\star\).

It remains to prove that \(x_\star\) is a KKT point.  Integrating
\eqref{eq:intro-flow} and dividing by \(t>0\) give
\begin{equation}\label{eq:flow-averaged-dual}
 \frac{\nabla\phi(x(t))-\nabla\phi(x(0))}{t}
 +\frac1t\int_0^t\nabla f(x(s))\,\dd s
 \in\range(A^\transpose).
\end{equation}
Since \(x(t)\to x_\star\), the integral average converges to
\(\nabla f(x_\star)\). The constraint qualification \(C\cap L\neq\varnothing\) implies that
\begin{equation}\label{eq:flow-CQ}
 N_{\overline C}(x_\star)\cap\range(A^\transpose)=\{0\}.
\end{equation}
Indeed, choose \(\widehat x\in C\cap L\).  If
\(v=A^\transpose\mu\in N_{\overline C}(x_\star)\), then
\(\ip{v}{\widehat x-x_\star}=0\).  For all sufficiently small
\(\varepsilon>0\), normality applied to
\(\widehat x+\varepsilon v\in C\) gives
\(0\geq\ip{v}{\widehat x+\varepsilon v-x_\star}
=\varepsilon\norm{v}^2\), proving \eqref{eq:flow-CQ}.

If \(x_\star\in C\), the first term in
\eqref{eq:flow-averaged-dual} tends to zero, so
\(\nabla f(x_\star)\in\range(A^\transpose)\) and the KKT condition follows.
Suppose that \(x_\star\) lies on the boundary.  Essential
smoothness gives \(\norm{\nabla\phi(x(t))}\to\infty\).  For all sufficiently
large \(t\), let
\[
 \nu(t):=\frac{\nabla\phi(x(t))-\nabla\phi(x(0))}
                 {\norm{\nabla\phi(x(t))}},
 \qquad
 \gamma(t):=\frac{\norm{\nabla\phi(x(t))}}{t}.
\]
Every cluster point of \(\nu(t)\) is a unit vector in
\(N_{\overline C}(x_\star)\)
\cite[Lemma~4.1]{alvarez2004hessian}.  The function \(\gamma\) is bounded.
Otherwise, along some \(t_j\to\infty\), one would have
\(\gamma(t_j)\to\infty\) and \(\nu(t_j)\to v\).  Dividing
\eqref{eq:flow-averaged-dual} by \(\gamma(t_j)\) would give
\(v\in\range(A^\transpose)\), contradicting \eqref{eq:flow-CQ}.

Choose \(t_j\to\infty\).  After passing to a subsequence, let
\(\gamma(t_j)\to\gamma_\star\geq0\) and \(\nu(t_j)\to v\).  Then
\(\gamma_\star v\in N_{\overline C}(x_\star)\).  Passing to the limit in
\eqref{eq:flow-averaged-dual} gives
\(\gamma_\star v+\nabla f(x_\star)\in\range(A^\transpose)\).  Hence there
exists \(\lambda_\star\in\R^m\) such that
\(0\in\nabla f(x_\star)+A^\transpose\lambda_\star
+N_{\overline C}(x_\star)\), which is \eqref{eq:KKT}.

\bibliographystyle{plain}
\begingroup
\footnotesize
\let\plainbibitem\bibitem
\renewcommand{\bibitem}[1]{\plainbibitem{#1}}
\bibliography{references}
\endgroup

\end{document}